\documentclass[reqno]{amsart}
 \usepackage{varioref}
 \usepackage{wrapfig}
 \usepackage{threeparttable}
 \usepackage{dcolumn}
  \newcolumntype{d}{D{.}{.}{-1}}
 \usepackage{nomencl}
  \makeglossary
 \usepackage{subfigure}
 \usepackage{subfigmat}
 \usepackage{fancyvrb}
  \fvset{fontsize=\footnotesize,xleftmargin=2em}
 \usepackage{lettrine}
 \usepackage[colorlinks]{hyperref}
 \usepackage{multicol}
 \usepackage{enumitem}
\hypersetup{
    colorlinks,
    linkcolor={red!70!black},
    citecolor={blue!70!black},
    urlcolor={blue!50!black}
}
\usepackage{graphicx}
\usepackage{psfrag}
\usepackage{tikz}
\usetikzlibrary{matrix}

\usepackage{amssymb}
\usepackage{amsthm}
\usepackage{amsrefs}

\usepackage{caption}

 \definecolor{ml1}{rgb}{    0,    0.4470,    0.7410}
 \definecolor{ml2}{rgb}{     0.8500   , 0.3250   , 0.0980}
 \definecolor{ml3}{rgb}{     0.9290   , 0.6940  ,  0.1250}
 \definecolor{ml4}{rgb}{     0.4940  ,  0.1840  ,  0.5560}
  \definecolor{ml5}{rgb}{    0.4660  ,  0.6740  ,  0.1880}
  \definecolor{ml6}{rgb}{    0.3010  ,  0.7450  ,  0.9330}
  \definecolor{ml7}{rgb}{    0.6350  ,  0.0780 ,   0.1840}
  
\renewcommand{\exp}[1]{{\bf e}^{#1}}
\newcommand{\expl}[1]{\textrm{exp} \left[#1\right]}

\newcommand{\T}{\mathsf T}

\newcommand{\M}{{\mathsf M}}
\newcommand{\RR}{{\mathbb R}}
\newcommand{\PP}{{\mathsf P}}
\newcommand{\Su}{{\mathsf S}}
\newcommand{\thetae}{\phi}

\newcommand{\xm}{{\bar x}}

\newcommand{\lb}{\left\{}
\newcommand{\rb}{\right\}}

\newcommand{\bb}[0]{\begin{bmatrix}}
\newcommand{\eb}[0]{\end{bmatrix}}
\newcommand{\be}[0]{\begin{equation}}
\newcommand{\ee}[0]{\end{equation}}
\newcommand{\ba}[0]{\begin{align}}
\newcommand{\ea}[0]{\end{align}}
\newcommand{\ben}[0]{\begin{equation*}}
\newcommand{\een}[0]{\end{equation*}}
\newcommand{\rep}[0]{\eqref}

\newcommand{\norms}[1]{\lVert#1\rVert}

\let\norm=\normB

\newcommand{\abs}[1]{\lvert#1\rvert}

\renewcommand{\Re}[0]{\mathbb R}

 \newtheorem{thm}{Theorem}
 \newtheorem{lem}[thm]{Lemma}
 
 \newtheorem{cor}[thm]{Corollary}
 \newtheorem{defn}{Definition}

 \newtheorem{rem}{Remark}

\usepackage{ulem}
\usepackage{etoolbox}
\usepackage{color}
\newtoggle{showcommentsone}

\settoggle{showcommentsone}{false}

\iftoggle{showcommentsone}{
\newcommand{\benone}[1]{{\color{blue}#1}}
}{
\newcommand{\benone}[1]{#1}
}
\iftoggle{showcommentsone}{
\newcommand{\bencrossone}[1]{{\color{red}\sout{#1}}}
}{
\newcommand{\bencrossone}[1]{}
}

\title[Convergence Properties of Adaptive Systems]{Convergence Properties of Adaptive Systems and the Definition of Exponential Stability}

  \author[B.M. Jenkins ]{Benjamin M. Jenkins}
 \address{PhD Candidate, Department of Mechanical Engineering MIT, Rm 3-441, 77 Massachusetts Avenue, Cambridge MA, 02139,USA}
 \email{bjenkins@mit.edu}

 \author[A.M. Annaswamy]{ Anuradha M. Annaswamy}
 \address{Director of Active-Adaptive Controls Lab, Department of Mechanical Engineering, MIT, 77 Massachusetts Avenue, Cambridge MA, 02139,USA } 
\email{aanna@mit.edu}

\author[E. Lavretsky]{Eugene Lavretsky}
\address{Senior Technical Fellow, Boeing Research and Technology, 5301 Bolsa Avenue, MC H017-D334, Huntington Beach, CA, 92648, USA}
\email{eugene.lavretsky@boeing.com}

 \author[T.E. Gibson]{Travis E. Gibson}
 \address{Research Fellow in Medicine, Harvard Medical School and Channing Laboratory, Brigham and Women's Hospital, 181 Longwood Avenue, Boston MA, 02115, USA}
 \email{tgibson@mit.edu}
 \email{travis.gibson@channing.harvard.edu}

\begin{document}

\maketitle

\begin{abstract}

The convergence properties of adaptive systems in terms of excitation conditions on the regressor vector are well known. With persistent excitation of the regressor vector in model \benone{reference} adaptive control the state error and the adaptation error are globally exponentially stable, or equivalently, exponentially stable in the large. When the excitation condition however is imposed on  the reference input or the reference model state it is often incorrectly concluded that the persistent excitation in those signals also implies exponential stability in the large. The definition of persistent excitation is revisited so as to address some possible confusion in the adaptive control literature. 
It is then shown that persistent excitation of the reference model only implies local persistent excitation (weak persistent excitation). Weak persistent excitation of the regressor is still sufficient for uniform asymptotic stability in the large, but not exponential stability in the large. We show that there exists an infinite region in the state-space of adaptive systems where the state rate is bounded. This infinite region with finite rate of convergence is shown to exist not only in classic open-loop reference model adaptive systems, but also in a new class of closed-loop reference model adaptive systems.


\end{abstract}

\section{Introduction}

It is well known that stability \benone{of the origin} and asymptotic convergence of the tracking error \benone{to zero} can be guaranteed in adaptive systems with no restrictions on the external reference input. Asymptotic stability, i.e. convergence of both the tracking error \benone{\textit{{and}}} parameter error \benone{to zero}, can occur only with further conditions of persistent excitation \benone{ are satisfied. }\bencrossone{ on the reference input. } The first known work on asymptotic stability of adaptive systems can be found in \cite{lio67}. In that work asymptotic stability of adaptive schemes was proven for a class of periodic inputs using results from \cite{las62}. The results hinged on a sufficient condition related to the richness of frequency content in the regressor vector of the adaptive system. In the late 70's and early 80's several attempts were made to extend the results of \cite{lio67} to uniform asymptotic stability. Morgan and Narendra proved necessary and sufficient conditions for uniform asymptotic stability for classes of linear time varying \benone{ (LTV) } systems  in \cite{mor77_1,mor77_2} that \benone{ are consistent with } the structure of adaptive systems. Anderson leveraged techniques developed in \cite{and69} to prove the exponential stability of adaptive systems in \cite{and77} with Kreisselmeier using similar techniques in \cite{kre77}. Following these results the persistent excitation conditions for asymptotic stability were moved from the regressor vector to \benone{ richness }conditions on the actual reference model input in references \cite{yua77,and82,boy83,boyd1986necessary,nar87}. This was a key step for practical reasons as the control engineer has direct control over the reference input rather than the regressor.


A key distinction exists, however, between the stability properties of the linear time-varying systems studied in \cite{and77} and those of the adaptive systems in \cite{mor77_2}, and forms the starting point for the discussions in this paper. The linear time-varying system in \cite{and77} can be shown to be {\it exponentially stable} under persistent excitation conditions on the underlying regressor. \benone{ However, once the excitation condition is moved to the reference input, the  }\bencrossone{ The } adaptive systems in \cite{mor77_2} can \benone{ only be shown to be \textit{uniformly asymptotically stable}. } \bencrossone{ however be shown to be only {\it uniformly asymptotically stable} under persistent excitation conditions on the external reference input. } \benone{ This distinction arises from the endogenous nature of the underlying regressor and is explicitly pointed out in this paper. The degree of persistent excitation of the regressor is dependent on the adaptive system initial conditions. This dependency prevents a uniform correlation between degree of persistent excitation (i.e. rate of exponential convergence) of the adaptive systems internal regressor and the richness of the reference input.  } \bencrossone{ This distinction is explicitly pointed out in this paper }\benone{ The practical implication of this is that the adaptive systems of \cite{mor77_2} are not  exponentially stable in the large. This distinction between local and global exponential stability is an essential detail when the exponential stability of a system is used to claim robustness properties. } \bencrossone{ where we will show in fact that the adaptive system cannot be shown to be exponentially stable in the large. } Moreover, an \bencrossone{ specific }  infinite region will be shown to exist \benone{ in the velocity field } where the norm of the error \bencrossone{ vector field } \benone{ velocity }is finite. As a result, a \benone{ subset of the state-space }{\bencrossone{  sticking regime }} will be shown to exist where the error signals move arbitrarily slowly. Unlike exponentially stable systems, the system's convergence speed decreases as the distance from the equilibrium increases.

 \benone{  }\benone{ No excitation (richness) conditions   on the reference input  exist which will globally guarantee persistent excitation of the systems internal regressor vector. }\bencrossone{ ersistence of excitation conditions on the reference input, also referred to as sufficient richness, does not imply persistence of excitation of the regressor vector. } The authors of \cite{boy83,boyd1986necessary} are careful in proving that richness of the reference input only implies exponential convergence. The careful wording of {{\textit{convergence}}} however was changed to exponential \textit{stability} countless times elsewhere in the literature. It was then inappropriately  concluded that uniform asymptotic stability in the large is equivalent to exponential stability in the large for adaptive systems.

Recently, a new class of adaptive systems has been under discussion (see \cite{gib13access,gib13ecc,gib12,gib13acc1,gib15tac,gibson_phd}) which employ a closed-loop in the underlying reference model. These adaptive
systems have \bencrossone{ a } desirable transient response \benone{ characteristics such as }\bencrossone{ which leads to } an improved tracking
error whose L-infinty and L-2 norms are small compared to their open-loop counterparts.
\benone{ In addition, the rates of closed-loop signals such as the control input and and control parameter have small magnitudes when compared to open-loop reference model adaptive systems. } \bencrossone{ More importantly, the closed-loop signals such as the control input and control
parameter have derivatives that have small magnitudes as well when compared to
open-loop reference model systems. } \benone{  } In reference \cite{jenkins2013convergence,nouwens_analysis_2015}, it was shown that the \bencrossone{ sticking
region } \benone{region of slow convergence }that is present in the standard adaptive system with Open-loop Reference
Models (ORM) \cite{jenkins2013uniform} is present \benone{ in this new class} of Closed-loop Reference Model (CRM)-based adaptive
systems as well.

This article is intended to be a cautionary \bencrossone{ tale }  \benone{ piece }and complements the works of \cite{panteley2001relaxed} and \cite{loria2002} in carefully defining persistent excitation and a weaker \bencrossone{ version } \benone{ condition }that is not uniform in initial conditions. Where as \cite{panteley2001relaxed} and \cite{loria2002} focus on the various stability results when two different kinds of persistent excitation are studied, we illustrate why, in general, adaptive systems can not satisfy the original \bencrossone{ strong version } \benone{ definition }of persistent excitation. \bencrossone{ Along that ideology  }We pick up where \cite{nar87} left off and in so doing hope to clarify the true stability properties of adaptive systems. We connect the stability results of general adaptive systems to the \bencrossone{ sticking } region \benone{ of slow convergence} in low dimensional adaptive systems that occur with ORM and CRM. The paper is organized as follows: Section 2 \bencrossone{ contains } \benone{ reviews the }definitions for various kinds of stability, Section 3 discusses the relationship between persistent excitation and asymptotic stability of adaptive systems, Section 4 constructs examples and proves the lack of exponential stability in the large \benone{ even }for \benone{low order }adaptive systems, Section 5 contains simulation results \benone{ which depict the nature of this slow convergence}, and Section 6 \benone{ summarizes our findings} \bencrossone{ contains the conclusion }.


\section{Stability Definitions}

Consider a dynamical system defined by the following relations
\begin{align*}
x(t_0) &= x_0 \\
\dot x(t) &= f(x(t),t)
\end{align*}
where ${t\in [t_0,\infty)}$ is time and $x\in \RR^n$ denotes the state vector. We are interested in systems with equilibrium at $x=0$, so that $f(0,t)=0$ for all $t$. The solution to the differential equation above for $t\geq t_0$ is a transition function $s(t;x_0,t_0)$ such that $\dot s(t;x_0,t_0)= f(s(t;x_0,t_0),t)$ and $s(t_0;x_0,t_0) = x_0$. Various definitions of stability now follow \cite{mas56,kal60,hah67}. Figure \ref{fig1} can be used as an aid.

\begin{figure}\subfigure[Stability and asymptotic stability visualized, adapted from \cite{kal60}*{Figure 5} ] {
  \psfrag{b}[cc][cc][.7]{$\delta, \rho$}
  \psfrag{a}[cc][cc][.7]{{${\epsilon}$}}
   \psfrag{t}[cc][cc][.7]{$t$}
    \psfrag{c}[cc][cc][.7]{{$\eta$}}
      \psfrag{d}[cc][cc][.7]{$x_0$}
     \psfrag{f}[cc][cc][.7]{{$t_0$}}
      \psfrag{h}[cl][cu][.7]{{$t_0+{T}$}}
        \psfrag{g}[cc][cc][.7]{$s$}
        \includegraphics[width=.5\textwidth]{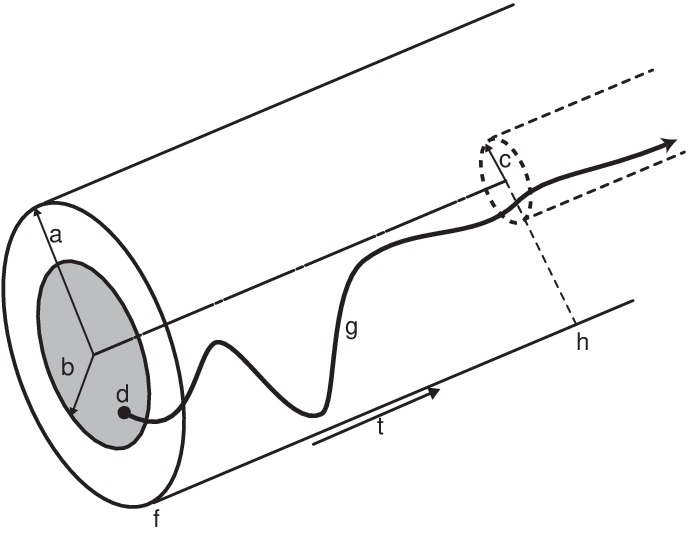}\label{fig1a}}
\subfigure[Exponential stability visualized.] {
  \psfrag{b}[cr][cr][.7]{{ ${\rho}$}}
  \psfrag{a}[cr][cr][.7]{{${\kappa}\norm{x_0}$}}
   \psfrag{t}[cc][cc][.7]{$t$}
    \psfrag{c}[cc][cc][.7]{${\eta}$}
      \psfrag{d}[cc][cc][.7]{$x_0$}
     \psfrag{f}[cc][cc][.7]{$t_0$}
      \psfrag{h}[cl][cu][.7]{$\kappa \norm{x_0} \exp{-\nu(t-t_0)}$}
        \psfrag{g}[cc][cc][.7]{$s$}
\includegraphics[width=.5\textwidth]{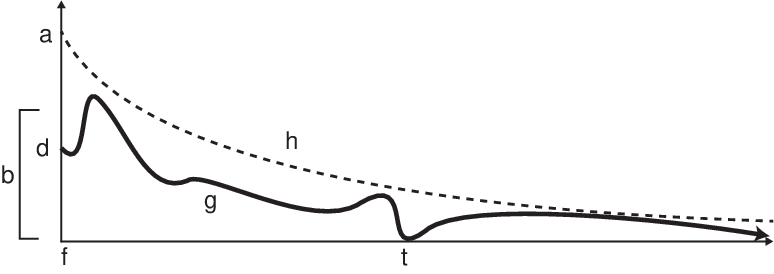}\label{fig2a}}%
\caption{Visual aids for stability discussion.}\label{fig1}
\end{figure}

\begin{defn}[Stability and Asymptotic Stability] Let $t_0\geq 0$, the equilibrium is
\label{def:stable} 
\begin{enumerate}
\item[\textnormal{(i)}] {\em Stable}, if for all $\epsilon>0$ there exists a $\delta(\epsilon,t_0)>0$ such that $\norm{x_0}\leq \delta$ implies $\norm{s(t; x_0, t_0)} \leq \epsilon$ for all $t \geq t_0$.
\item[\textnormal{(ii)}] {\em Attracting}, if there exists a $\rho(t_0)>0$ such that for all $\eta>0$ there exists an attraction time $T(\eta,x_0,t_0)$ such that $\norm{x_0}\leq\rho$ implies $\norm{s(t;x_0,t_0)} \leq\eta$ for all $t\geq t_0+T$.
\item[\textnormal{(iii)}] {\em Asymptotically Stable}, if it is stable and attracting.
\item[\textnormal{(iv)}] {\em Uniformly Stable} if the $\delta$ in {\rm (i)} is uniform in $t_0$ and $x_0$, thus taking the form $\delta(\epsilon)$.
\item[\textnormal{(v)}] {\em Uniformly Attracting}, if it is attracting where the $\rho$ and $T$ do not depend on $t_0$ or $x_0$ and thus the attracting time take the form $T(\eta,\rho)$.
\item[\textnormal{(vi)}] {\em Uniformly Asymptotically Stable},  {\upshape (UAS)} if it is uniformly stable and uniformly attracting.
\item[\textnormal{(vii)}] {\em Uniformly Bounded} if for all $r>0$ there exists a $B(r)$ such that $\norm{x_0}\leq r$ implies that $\norm{s(t; x_0, t_0)}\leq B$ for all $t\geq t_0$.
\item[\textnormal{(viii)}] {\em Uniformly Attracting in the Large} if for all $\rho>0$ and $\eta>0$ there exists a  $T(\eta,\rho)$ such that $\norm{x_0}\leq\rho$ implies $\norm{s(t;x_0,t_0)} \leq\eta$ for all $t\geq t_0+T$.
\item[\textnormal{(ix)}] {\em Uniformly Asymptotically Stable in the Large}  {\upshape (UASL)} if it is uniformly stable, uniformly bounded, and  uniformly attracting in the large.
\end{enumerate}
\end{defn}

The definitions of exponential stability are not as prevalent as those above, and are assembled below  \cite{mas56, malkin35}


\begin{defn}[Exponential Stability] Let $t_0\geq 0$, the equilibrium is
\label{def:exp} 
\begin{enumerate}
\item[\textnormal{(i)}] {\em Exponentially Stable} {\upshape (ES)} if for every $\rho>0$ there exists ${\nu(\rho)>0}$ and ${\kappa(\rho)>0}$ such that $\norm{x_0}\leq \rho$ implies $\norm{s(t;x_0,t_0)} \leq \kappa \norm{x_0} \exp{-\nu(t-t_0)}$
\item[\textnormal{(ii)}] {\em Exponentially Stable in the Large}  {\upshape (ESL)} if there exists ${\nu>0}$ and ${\kappa>0}$ such that $\norm{s(t;x_0,t_0)} \leq \kappa \norm{x_0} \exp{-\nu(t-t_0)}$ for all $x_0$.
\end{enumerate}
\end{defn}


\begin{rem}\label{rem:uasl}ESL implies UASL by choosing $T(\rho,\eta)=\frac{1}{\nu}\log\bigl(\frac{\kappa \rho}{\eta}\bigr)$.
It is clear that for UASL, $T$ is a function of both $\rho$ and $\eta$. But for ESL, $T$ depends only on $\eta/\rho$. In other words, if in a system, it can be shown that $T$ is a general function of $\eta$ and $\rho$ and varies even when $\eta/\rho$ is a constant, then it follows that the associated equilibrium is only UASL and not ESL.
 \end{rem}

For linear systems, i.e. $\dot x = A(t)x$,  UAS implies ESL \cite[Theorem 3: (C) and (D)]{kal60}. Thus, for linear systems all of the definitions are equivalent.
The relationship between these definitions of stability are illustrated in the following implication diagram. 


{\centering
\begin{tikzpicture}
  \matrix (m) [matrix of math nodes,row sep=3em,column sep=3em,ampersand replacement=\&]
  {\text{ESL} \& \text{ES}    \& \\
             \&        \text{UASL}         \& \text{UAS}  \\};
  \path[-stealth]
    (m-1-1) edge [double]  (m-1-2)
    (m-1-2) edge [double]  (m-2-2)
    (m-1-2) edge [double]  (m-2-3)
    (m-2-2) edge [double]  (m-2-3)
    (m-2-3) edge [double,bend left=45] node[below left] {\footnotesize{+ Linear}} (m-1-1);
\end{tikzpicture} \\ }


\section{Asymptotic and Exponential Stability of Adaptive Systems}

We now present two adaptive systems which arise in the context of identification and control. The following definition of persistent excitation is relevant for exponential stability of adaptive systems. 
\begin{defn}[Persistent Excitation]  \label{defn:pe}Let $\omega\in [t_0,\infty)\to \Re^p$ be a time varying parameter with initial condition defined as $\omega_0 = \omega(t_0)$, then the parameterized function of time ${y(t,\omega):  [t_0,\infty) \times \Re^p \to \Re ^m}$ is
\label{def:pe} 
\begin{enumerate}
\item[\normalfont{(i)}] {\em Persistent Excitation}  $(\mathrm{PE})$ if there exists  ${T>0}$  and $\alpha>0$ such that \ben\int_t^{t+T} y(\tau,\omega) y^\mathsf{T} (\tau,\omega) d\tau \succeq \alpha I \een for all $t\geq t_0$ and $\omega_0 \in \Re^p$, and we denote this as $y(t,\omega)\in \mathrm{PE}$.
\item[\normalfont{(ii)}] {\em weak Persistent Excitation} $(\mathrm{PE}^*(\omega,\Omega))$ if there exists  a compact set $\Omega\subset \Re^p$, $T(\Omega)>0, \alpha(\Omega)$ such that \ben\int_t^{t+T} y(\tau,\omega) y^\mathsf{T}(\tau,\omega) d\tau \succeq \alpha I \een for all ${\omega_0\in\Omega}$ and ${t\geq t_0}$, and we denote this as ${y(t,\omega)\in\mathrm{PE}^*(\omega,\Omega)}$.
\end{enumerate}
\end{defn}
\noindent The PE definition is well known in the literature \cite{annbook,ioabook,nar87}, while the weak PE, denoted as PE$^*$, is introduced in this paper, and will be used to characterize convergence in adaptive systems.

\subsection{Identification in Simple Algebraic Systems \cite[]{annbook}}
Let $u:\mathcal [t_0,\infty)  \to \Re^n$ be the input and $y:[t_0,\infty) \to \Re$ be the output of the following algebraic system of equations
\ben
y(t) = u^\mathsf{T}(t) \theta 
\een
where $\theta\in \Re^n$ is an unknown parameter. If we assume that $u$ is known and $y$ is measurable, then an estimate of the unknown parameter $\hat \theta:[t_0,\infty)\to \Re^n$ can be used in constructing an adaptive observer
\ben\hat y (t)= u^{\mathsf T}(t) \hat \theta(t)
\een
where the update for the estimate of the uncertain parameter is defined as
\ben
\dot {\hat \theta}(t) = - u(t) \left( \hat y(t) - y(t) \right) .
\een
Denoting the parameter error as $\phi(t) = \hat \theta(t)- \theta$ the parameter error evolves as
\be\label{eq:algebraadaptive}
\dot {\phi}(t) = - u(t) u^\mathsf{T}(t) \phi(t).
\ee

\begin{thm}\label{thm:penar}
If $u(t)$  is PE, piecewise continuous, and either 1) there exists $\beta>0$ such that
\ben\int_t^{t+T} u(\tau) u^\mathsf{T} (\tau) d\tau \preceq \beta I  \een
or 2) there exists a $u_\mathrm{max}>0$ such that $\norm{u(t)}\leq u_\mathrm{max}$, then for the dynamics in  \eqref{eq:algebraadaptive} the equilibrium $\phi=0$ is ESL.
\end{thm}
The proof is given in two flavors the first follows that of \cite{and77} and the second follows that of \cite{annbook}, and then the two methods are compared.\begin{proof}[Proof of the theorem following Anderson {\cite[proof of Theorem 1]{and77}}]
 The existence of $T$, $\alpha$, and $\beta$ such that $ \alpha I \preceq \int_t^{t+T} u(\tau) u^\mathsf{T} (\tau) d\tau\preceq\beta I$ is equivalent to the following system being {\em uniformly completely observable} $\Sigma_1:\ \dot x_1 = 0_{n\times n} x_1, \ y_1 = u^{\mathsf T}(t) x_1$ \cite[Definition (5.23) dual of (5.13)]{kal_contr}. This in turn implies that ${\Sigma_2}:{\dot x_2 = -u(t)u^{\mathsf T}(t)x_2}, \ {y_2 = u^{\mathsf T}(t) x_2}$ is uniformly completely observable as well  \cite[Dual of Theorem 4]{and69}. Therefore, there exists $\alpha_2$ and $\beta_2$ such that 
\be\label{eq:proof:algebra1}
 \alpha_2 I \preceq\int_t^{t+T}  \Phi_2^{\mathsf T}(\tau,t)u(\tau) u^\mathsf{T} (\tau)  \Phi_2(\tau,t) d\tau \preceq \beta_2 I
\ee
where $\Phi_2(t,t_0)$ is the state transition matrix for $\Sigma_2$. Note that the upper bound $\beta$ is needed to ensure that $\Phi_2(\tau,t)$ is not singular, $$\det \Phi_2 (t,t_0) = \expl{-\int_{t_0}^{\mathsf T} \textrm{trace}(u(\tau)u^\mathsf{T}(\tau))\, d\tau}.$$

Let $V(\phi,t)=\frac{1}{2}\phi^{\mathsf{T}}(t)\phi(t)$ and note that $\Sigma_2$ and \eqref{eq:algebraadaptive} have the same state transition matrix. Thus $\phi(t;t_0)=\Phi_2(t,t_0) \phi(t_0)$. Differentiating $V$ along the system trajectories in \eqref{eq:algebraadaptive} we have $\dot V(\phi,t;t_0) = -\phi^\T(t_0) \Phi_2^{\mathsf T}(t,t_0)u(t) u^\mathsf{T} (t)  \Phi_2(t,t_0) \phi(t_0)$.  Using the bound in \eqref{eq:proof:algebra1} and integrating as $\int_t^{t+T} \dot V(\phi,\tau;t) d\tau $, it follows that ${V(t+T) -V(t)} \leq -2 \alpha_2V(t)$. Thus $V(t+T) \leq (1-2\alpha_2) V(t)$ and therefore the system is UASL and due to linearity it follows that the systems is ESL.
\end{proof}
\begin{proof}[Proof of the theorem following Narendra and Annaswamy {\cite[proof of Theorem 2.16]{annbook}}]
First we note that $u(t)$ being PE is equivalent to
\ben
\int_{t}^{t+T} \abs{u^\T(\tau) w}^2 d\tau \geq \alpha
\een
holding for any fixed unitary vector $w$. Let $\tilde u(t) \triangleq \frac{u(t)}{u_\mathrm{max}}$, then it follows that 
\ben\begin{split}
\int_{t}^{t+T} \abs{u^\T(\tau) w}^2 d\tau &= u_\mathrm{max}^2 \int_{t}^{t+T} \abs{\tilde u^\T(\tau) w}^2 d\tau\\
&\leq  u_\mathrm{max}^2 \int_{t}^{t+T} \abs{\tilde u^\T(\tau) w}d\tau
\end{split}\een
where the second line of the above inequality follows due to the fact that $\norm {\tilde u}\leq 1$ and thus $\abs{\tilde u^\T(\tau) w}^2 \leq \abs{\tilde u^\T(\tau) w} $. Therefore, $u$ being PE and bounded implies that 
\be\label{eq:puu1}
\frac{\alpha}{u_\mathrm{max}} \leq \int_{t}^{t+T} \abs{ u^{\mathsf T}(\tau) w} d\tau.
\ee
The above bound will be called upon shortly. Moving forward with the proof, consider the Lyapunov candidate $V(\phi,t)=\frac{1}{2}\phi^{\mathsf T}(t)\phi(t)$. Then differentiating along the system directions it follows that $\dot V(\phi,t)= - \phi^{\mathsf T}(t)u(t) u^{\mathsf T}(t) \phi(t)$. Integrating $\dot V$ and using the Cauchy Schwartz inequality it follows
\ben\begin{split}
-\int_t^{t+T}\dot V(\phi,\tau) d\tau & = \int_t^{t+T}  \abs{ u^{\mathsf T}(\tau) \phi(\tau)}^2 d\tau\\  &\geq \frac{1}{T}\left( \int_t^{t+T}  \abs{ u^{\mathsf T}(\tau) \phi(\tau)} d\tau \right)^2.
\end{split}\een
The above inequality can equivalently be written as 
\be\label{eq:puu2}
 \sqrt{{T}({V(t)-V(t+T)})} \geq \int_t^{t+T}  \abs{ u^{\mathsf T}(\tau) \phi(\tau)} d\tau.
\ee
Using the reverse triangle inequality, the righthand side of the inequality in \eqref{eq:puu2} can be bounded as 
\begin{equation}\label{eq:puu3}
  \int_t^{t+T} \abs{ u^{\mathsf T}(\tau) \phi(\tau)} d\tau \geq \int_t^{t+T}  \abs{ u^{\mathsf T}(\tau) \phi(t)} d\tau - 
  \int_t^{t+T}  \abs{ u^{\mathsf T}(\tau) [\phi(t)-\phi(\tau)]} d\tau.
\end{equation}
Using the bound in \eqref{eq:puu1} the first integral on the righthand side of the above inequality can be bounded as
\be\label{eq:puu4}
\int_t^{t+T}  \abs{ u^{\mathsf T}(\tau) \phi(t)} d\tau \geq \norm{\phi(t)} \frac{\alpha}{u_\mathrm{max}}.
\ee
The second integral on the righthand side of \eqref{eq:puu3} can be bounded as
\be
\begin{split}\label{eq:puu5}
\int_t^{t+T}  \abs{ u^{\mathsf T}(\tau) [\phi(t)-\phi(\tau)]} d\tau  &\leq  u_\mathrm{max} T \sup_{\tau \in [t,t+T]}\norm{\phi(t)-\phi(\tau)}\\  & \leq  u_\mathrm{max} T \int_t^{t+T}  \norm{\dot \phi(\tau)} d\tau \\
& \leq u_\mathrm{max}^2 T \int_t^{t+T}  \norm{u ^\mathsf{T}(\tau) \phi(\tau)} d\tau.
\end{split}
\ee
The second line in the above inequality follows by the fact that the arc-length between two points in space is always greater than or equal to a strait line between them. The third line in the above inequality follows by substition of the dynamics in \eqref{eq:algebraadaptive}. Substitution of the inequalities in \eqref{eq:puu3}-\eqref{eq:puu5} into \eqref{eq:puu2} it follows that 
\ben
 \int_t^{t+T} \abs{ u^{\mathsf T}(\tau) \phi(\tau)} d\tau \geq \frac{\norm{\phi(t)} \frac{\alpha}{u_\mathrm{max}}}{1+u_\mathrm{max}^2 T}.
\een
Substitution of the above bound into \eqref{eq:puu2} and squaring both sides it follows that 
\ben
V(t+T) \leq \left(1- \frac{2 {\alpha^2}/{u_\mathrm{max}^2 }}{T (1+u_\mathrm{max}^2 T)^2} \right) V(t).
\een
Therefore the dynamics in \eqref{eq:algebraadaptive} are UASL and by linearity this implies ESL as well.
\end{proof}

While the first proof is more generic, the method deployed in the second proof gives direct insight as to how the degree of persistent excitation, $\alpha$, and the upper bound,  $u_\mathrm{max}$, affect the rate of convergence,
\be\label{eq:rcon}
 r_\text{con} \triangleq \left(1- \frac{2 {\alpha^2}/{u_\mathrm{max}^2 }}{T (1+u_\mathrm{max}^2 T)^2} \right) .
\ee In the method by Anderson the rate of convergence is an existence one given by $ (1-2\alpha_2)$. No closed form expression is given relating $\alpha_2$ to the original measures of PE, $\alpha$ and $\beta$.\footnote{If one carefully follows the steps outlined in \cite{and77} it may be possible to come up with a closed form relation, but it appears to be non-trivial.} It is clear however that for fixed $T$ an increase in $u_\mathrm{max}$ conservatively implies an increase in $\beta$. It is also clear from \eqref{eq:rcon} that an increase in $u_\mathrm{max}$ decreases the convergence rate $ r_\text{con}$. We show below that an increase in $\beta$ implies a decrease in $r_\text{con}$. Recall the Abel-Jacobi-Liouville identity, $\det \Phi_2 (t,t_0) = \expl{-\int_{t_0}^{\mathsf T} \textrm{trace}(u(\tau)u^\mathsf{T}(\tau))\, d\tau}$, and thus as $\beta$ increases, $\det \Phi_2 (t,t_0)$ decreases. Now 
using this fact and the bound in \eqref{eq:proof:algebra1} it follows that as $\beta$ increases $\alpha_2$ decreases. 



Often, adaptive systems generate a dynamic system of the form \eqref{eq:algebraadaptive} where $u(\cdot)$ is a function of the parameter estimate itself. For this purpose, a nonlinear system of the form
\be\label{eq:algebraadaptive2}
\dot {\phi}(t) = - u(t,\phi) u^\mathsf{T}(t,\phi) \phi(t)
\ee
with $\phi_0=\phi(t_0)$ needs to be analyzed. This is addressed in following Theorem\benone{ , where it should be noted that UASL\ does not imply ESL}.  
\begin{thm}
Let $\Omega(r) = \{\phi : \norm {\phi} \leq r\}$. If $u(t,\phi) \in \mathrm{PE}^*(\phi,\Omega(r))$ for all $r$, $u(t)$ is piecewise continuous, and there exists $u_\mathrm{max}(r)>0$ such that $\norm{u(t,\phi)}\leq u_\mathrm{max}$ for all $\phi_0\in\Omega(r)$, then $\phi$ in Equation \eqref{eq:algebraadaptive2} is UASL\bencrossone{ and it {\em does not follow} that \eqref{eq:algebraadaptive2} is ESL }.
\end{thm}
\begin{proof} Given that $u(t,\phi) \in \textrm{PE}^*(\phi,\Omega(r))$, it follows that there exists $T(r)$ and $\alpha(r)$ such that $\int_{t}^{t+T} \abs{u^\T(\tau,\phi) w}^2 d\tau  \succeq  \alpha$ for all $\phi_0\in\Omega(r)$.  Choosing a Lyapunov candidate as  $V(\phi,t)=\frac{1}{2}\phi^{\mathsf T}(t)\phi(t)$ and following the same steps as in the proof of Theorem \ref{thm:penar} it follows that $V(t+T(r))\leq r_{\text{con}}V(t)$ for all $\phi_0\in\Omega(r)$ where 
\ben
r_\text{con}(r)= \left(1- \frac{2 {\alpha^2(r)}/{u_\mathrm{max}^2(r) }}{T(r) (1+u_\mathrm{max}^2(r) T(r))^2} \right).
\een
Given that the convergence rate is upper bounded for all $\norm{\phi_0}\leq r$ and $r$ can be arbitrarily large, the dynamics in \eqref{eq:algebraadaptive2} are UASL. In order for one to conclude that the dynamics are ESL there would need to exist a constant $0<\delta<1$ such that $r_\text{con}\leq \delta$ for all $r$. That is, the convergence rate of the Lyapunov function would need to be bounded away from 1 uniformly in initial conditions. This global uniformity is not achievable with this analysis and thus it is not possible to conclude ESL.
\end{proof}

\noindent In the next section we present an application of adaptive control where systems of the form \eqref{eq:algebraadaptive2} occur.


\subsection{Model Reference Adaptive Control}\label{sec:orm}
Let $u:\mathcal [t_0,\infty)  \to \Re$ be the input and $x:[t_0,\infty) \to \Re^n$ the state of a dynamical system
\be
\dot x(t) = A x(t) - B\theta^\T x(t) + Bu(t) \label{eq:xmain}
\ee
where $A\in\Re^{n\times n}$ is known and Hurwtiz and $B\in\Re^{n}$ is known as well, with the parameter $\theta\in\Re^n$ unknown. The goal is to design the input so that $x$ follows a reference model state $x_m:[t_0,\infty) \to \Re^n$ defined by the linear system of equations
\ben
\dot x_m(t) = A x_m(t)+ Br(t)
\een
where $r:[t_0,\infty) \to \Re$ is the reference command. Defining the model following error as $e=x-x_m$ the control input $u(t) =\hat\theta^\mathsf{T}(t)x(t) + r(t)$ achieves this goal when the adaptive parameter $\hat\theta: [t_0,\infty) \to \Re^n$ is updated as follows
\ben
\dot{\hat{\theta}}(t)= -xe^\mathsf{T}PB
\een
where $P=P^\T\in\Re^{n\times n}$ is the positive definite solution to the Lyapunov equation $A^\mathsf{T} P+PA = -Q$ for any real $n\times n$ dimensional $Q=Q^\T \succ0$. So as to simplify the notation we let $C \triangleq PB$ and the adaptive system can be compactly represented as
\be\label{eq:zdyn1}
\bb \dot e(t) \\ \dot \phi(t) 
\eb = \bb A  &  B x^\T(t)\\
               -x(t) C^\T & 0 
               \eb \bb e(t) \\ \phi(t)  \eb
\ee
where the initial conditions of the model following error and parameter error are denoted as $e_0=e(t_0)$ and $\phi_0=\phi(t_0)$. For the dynamics of interest it follows that \be\label{eq:lappy3} V(e,\phi)=e^\T Pe+\phi^\T\phi\ee is a Lyapunov candidate with time derivative along the state trajectories satisfying the inequality, $\dot V \leq - e^\T Q e$. This implies that $e(t)$ and $\phi(t)$ are bounded for all time with 
\be\label{eq:e}\norm{e} \leq \sqrt{\tfrac{V(e_0,\phi_0)}{P_\mathrm{min}}} \text{ and } \norm{\phi} \leq \sqrt{V(e_0,\phi_0)}\ee where $P_\mathrm{min}$ is the minimum eigenvalue of $P$. The reference command is bounded by design and thus $x_m$ is bounded and along with the bounds above implies that $x$ is bounded. The boundedness of $x$ and $\phi$ in turn implies that  $\dot e$ is bounded for all time.
Integration of $\dot V$ shows that $e\in\mathcal L_2$ with 
\be\label{eq:el2}
\norm{e}_{\mathcal L_{2}} \leq \sqrt{\tfrac{V(e_0,\phi_0)}{Q_\mathrm{min}}}
\ee
where $Q_\mathrm{min}$ is the minimum eigenvalue of $Q$. From the fact that $e\in \mathcal L_2\cap \mathcal L_\infty$ and ${\dot e} \in \mathcal L_\infty$ it follows that $e\to0$ as $t\to\infty$ \cite[Lemma 2.12]{annbook}. Before discussing the asymptotic stability of the dynamics in \eqref{eq:zdyn1} the following lemma is critical in relating persistent excitation between the reference model state and the plant state. Let $z=[e^\T,\,\phi^\T]^\T$, then the dynamics in \eqref{eq:zdyn1} can be compactly expressed as
\be\label{eq:zdyn2}
\dot z(t) = \bb A  &  B x^\T(t,z;t_0)\\
               -x(t,z;t_0) C^\T & 0 
               \eb z(t)
\ee
where we have explicitly denoted $x$ as a function of the state variable $z$.

\begin{lem}\label{lem:main}
For the dynamics in \eqref{eq:zdyn2} if $x_m(t)$ is PE with an $\alpha$ and $T$ such that $\int_t^{t+T} x_m(\tau) x_m^\mathsf{T} (\tau) d\tau  \succeq  \alpha I$,  and there exists a $\beta$ such that $\norm{x_m(t)}\leq \beta$,  then $x(t,z)$ is $\mathrm{PE}^*(z,Z(\zeta))$ with $Z(\zeta) = \{z: V(z)\leq \zeta\}$ for all $\zeta>0$  with the following bounds holding  
\be\label{eq:pxm}
\int_t^{t+pT} x(\tau) x^\mathsf{T} (\tau) d\tau  \succeq  \alpha^\prime I \ee
with $p>p_\mathrm{min}$ where
\be\label{eq:pmin}
\sqrt{p_\mathrm{min}} \triangleq \frac{\left(\sqrt{\tfrac{\zeta}{P_\mathrm{min}}}+2 \beta\right) \sqrt{T  \tfrac{\zeta}{Q_\mathrm{min}}}}{\alpha}
\ee
and 
\be
\alpha^\prime \triangleq p \alpha - \left(\sqrt{\tfrac{\zeta}{P_\mathrm{min}}}+2 \beta\right) \sqrt{p T  \tfrac{\zeta}{Q_\mathrm{min}}}.\ee
\end{lem}

Before going to the proof of this lemma a few comments are in order. First, note that the state variable $z$ contains both the model following error $e$ and the parameter error $\phi$. Therefore, what is being said is that there is weak persistent excitation of $x$ for all initial conditions $e_0$ and $\phi_0$ in the compact regions defined by the level sets of the Lyapunov function $V(z)=e^{\mathsf T}Pe + \phi^{\mathsf T}\phi $. Furthermore, because these conditions hold for arbitrarily large level sets, i.e. $\zeta$ can be arbitrarily large, weak persistent excitation of $x$ is achieved for  any initial condition $z_0\in\Re^{2n}$. However, because the parameters in the  persistent excitation bound in \eqref{eq:pxm}, namely $p$, are not uniform in $z_0$ it can not be concluded that $x$ is PE.
\begin{proof} This proof follows closely that of \cite[Theorem 3.1]{boy83}. For any fixed unitary vector $w$, consider the following equality, $ (x_m^\T w)^2 - (x^\T w)^2 = - (x^\T w - x_m^\T w)(x^\T w+x_m^\T w)$. Using the definition of $e$, the bound in \eqref{eq:e} for $e$ and the bound $\beta$ in the statement of the lemma, it follows that
\ben
 (x_m^\T w)^2 - (x^\T w)^2  \leq \norm{e} \left(\sqrt{\tfrac{V(z_0)}{P_\mathrm{min}}}+2 \beta\right).
\een
Moving $ (x_m^\T w)^2$ to the righthand side, multiplying by $-1$ and integrating from $t$ to $t+pT$ where $p$ is defined just above \eqref{eq:pmin}
\begin{equation*}
\int_t^{t+pT}(x^\T(\tau)  w)^2 d\tau \geq \int_t^{t+pT}  (x_m^\T(\tau)  w)^2d\tau - 
\left(\sqrt{\tfrac{V(z_0)}{P_\mathrm{min}}}+2 \beta\right)\int_t^{t+pT}\norm{e(\tau)} d\tau.
\end{equation*}
Applying Cauchy-Schwartz to the integral on the right hand side and using the fact that $\int_t^{t+T} (x_m^\T(\tau)  w)^2 d\tau \geq \alpha$ we have that
\begin{equation*}
\int_t^{t+pT}(x^\T(\tau)  w)^2 d\tau \geq p \alpha  - 
\left(\sqrt{\tfrac{V(z_0)}{P_\mathrm{min}}}+2 \beta\right) \sqrt{pT  \int_t^{t+pT}\norm{e(\tau)}^2 d\tau }.
\end{equation*}
Applying the bound in \eqref{eq:el2} for the $\mathcal L_2$ norm of $e$, it follows that
\ben
\int_t^{t+pT}(x^\T(\tau)  w)^2 d\tau \geq p \alpha- 
\left(\sqrt{\tfrac{V(z_0)}{P_\mathrm{min}}}+2 \beta\right) \sqrt{pT  \tfrac{V(z_0)}{Q_\mathrm{min}}}.
\een
For all $z_0\in Z(\zeta)$ it follows that $V(z_0)\leq \zeta$ and therefore 
\ben p \alpha - \left(\sqrt{\tfrac{V(z_0)}{P_\mathrm{min}}}+2 \beta\right) \sqrt{pT  \tfrac{V(z_0)}{Q_\mathrm{min}}}\geq\alpha^\prime.\een
It follows directly that $\int_t^{t+pT}(x^\T(\tau)  w)^2 d\tau \geq \alpha^\prime$ for all $t\geq t_0$ and $z_0\in Z(\zeta)$.
\end{proof}
\begin{rem} \label{rem:1}
The main take away from this lemma is that for a given $\alpha$ and $T$ such that $\int_t^{t+T} x_m(\tau) x_m^\mathsf{T} (\tau) d\tau  \succeq \alpha I$ and for a fixed $\alpha^\prime$ such that $\int_t^{t+pT} x(\tau) x^\mathsf{T} (\tau) d\tau  \succeq \alpha^\prime I$, as the size of the level set $V(z)=\zeta$ is increased, $p$ must also increase. This can be seen directly through \eqref{eq:pmin} where  $p_\mathrm{min}$ increases with increasing $\zeta$. Thus, as $p$ increases, the time window $pT$ over which the excitation is measured  increases as well.
\end{rem}

\begin{thm}\label{thm:ormpe}
 If $r(t)$ is piecewise continuous and bounded, and $x_m(t)$ is {PE} and uniformly bounded, then the the equilibrium of the dynamics in \eqref{eq:zdyn2} is UASL\bencrossone{  and it {\em does not follow} that it is ESL }.
\end{thm}
\begin{proof}
Given that $x_m\in\textrm{PE}$ it follows from Lemma \ref{lem:main} that $x(t,z)\in\textrm{PE}^*(z,Z(\zeta))$ for any $\zeta$ where  $Z(\zeta) = \{z: V(z)\leq \zeta\}$ and the Lyapunov function $V$ is defined in \eqref{eq:lappy3}. From  \eqref{eq:e} it follows that all signals are bounded. Furthermore given that $r$ is piecewise continuous and bounded it follows from \eqref{eq:xmain} that $\dot x$ is piecewise continuous. Therefore ${x\in\mathcal P_{[t_0,\infty)}}$, see  Definition \ref{def:P} of {\em piecewise smooth} in the Appendix. With $x(t,z)\in\textrm{PE}^*(z,Z(\zeta))\cap \mathcal P_{[t_0,\infty)}$ for any fixed $\zeta$ applying \cite[Theorem 5]{mor77_1} it follows that the dynamics of interest are UAS. Given that the above results hold for any $\zeta>0$, the dynamics of interest are therefore UASL. Due to the fact that persistent excitation bounds for $x$ do not hold globally uniformly in the initial condition $z_0$ one is not able to conclude ESL from this analysis. 
\end{proof}

We can in fact state something even stronger, and will give a proof by example in the following section (following Theorem \ref{thm:orm}).
\begin{thm}\label{thm:orms}
 The reference command $r(t)$ being  piecewise continuous and bounded, and the reference model state $x_m(t)$ being uniformly bounded and {PE} are not  sufficient  for the equilibrium of the dynamics in \eqref{eq:zdyn2} to be ESL.
 \end{thm}

\section{Lack of Exponential Stability in the Large for Adaptive Systems}



In this section two examples are presented so as to illustrate rigorously by example the implication made in Theorem \ref{thm:ormpe}, i.e. persistent excitation of the reference model does not imply exponential stability \benone{   in the large} of the adaptive system and thus proves Theorem \ref{thm:orms}. This is performed by constructing an invariant unbounded region in the state space of the direct adaptive system where the rate \benone{ of change per unit time} of \bencrossone{ convergence} \benone{ the system state }\bencrossone{ to the origin} is finite. It is this feature which implies a lack of exponential stability in the large.

The first example is identical to the dynamics in \eqref{sec:orm}, but with a learning gain added to the update law. The second example is a modified version of classic direct adaptive control with an error feedback term in the reference model \cite{gib13access,gibson_phd}. So as to distinguish the two systems we characterize them by their reference models and refer to the first system as {\em Open-loop Reference Model} (ORM) adaptive control and the second as {\em Closed-loop Reference Model} (CRM) adaptive control. The CRM adaptive system has been added due to recent interest in transient properties of adaptive systems, with the class of CRM systems portraying smoother trajectories as compared to their ORM counterpart, \cite{gib13access,gibson_phd}.

\subsection{Scalar ORM Adaptive Control with PE Reference State}


The following scalar dynamics are nearly identical to those in Section \ref{sec:orm} however we repeat them herein with $A=a<0$ and $B=b>0$ to emphasize that they are scalars. Let $u:\mathcal [t_0,\infty)  \to \Re$ be the input, $x:[t_0,\infty) \to \Re$ the plant state, $x_m:[t_0,\infty) \to \Re$  the reference state and $r:[t_0,\infty) \to \Re$ the reference input to the following set of differential equations
\begin{align}
\dot x(t) &= a x(t) - b\theta x(t) + bu(t) \label{eq:orm1}\\
\dot x_m(t) &= a x_m(t)+ br(t), \label{eq:orm2}
\end{align}
 with the parameter $\theta\in\Re$ unknown. 
For ease of exposition, in both this section and the next, we will assume that $r(t)$ is a non-zero constant, i.e., $r(t) \equiv \overline r$, $\overline r\neq 0$. The control input  is defined as $u(t) =\hat\theta(t)x(t) + r(t)$ with $\hat\theta: [t_0,\infty) \to \Re$ updated as follows
\be\label{eq:orm3}
\dot{\hat{\theta}}(t)= -\gamma xe,
\ee
where $e=x-x_m$ and  $\gamma>0$ is a tuning gain. 

As before, the error dynamics can be compactly expressed in vector form as ${z(t)= [e(t),\thetae(t)]^{\mathsf T}}$. A sufficient condition for the uniform asymptotic stability of the above system is that the reference input remain a non-zero constant for all time. This can be proved using Theorem \ref{thm:ormpe}. Given that for a constant reference command the above dynamics are also autonomous we give the same result using the well known invariance principle from Lasalle and Krasovskii \cite{kra63,1086720,bar52nauk}.

\begin{thm}\label{thm:orm1}
For the system defined in Equations {\rep{eq:orm1}-\rep{eq:orm3}} and $r(t) \equiv \overline r$, $\overline r\neq 0$ $z=0$ is UASL.\end{thm}
\begin{proof}
Define the Lyapunov function
\be\label{eq:lyaporm}
V(e,\thetae) = e^2 + \frac{1}{\gamma}\thetae^2
\ee
then
$
\dot V(e,\thetae) = 2a e^2
$.
Since $V > 0$ for all $z\neq 0$, $\dot V \leq 0$ for all $z \in \RR^2$ and $V \to \infty \; as \;z \to \infty$ the equilibrium at the origin is uniformly stable and uniformly bounded.  Given that the system is autonomous, it follows from the the invariance principle that the origin is UASL. 
\end{proof}


We are now going to construct an unbounded invariant region as discussed at the beginning of this section. The reference model state initial condition is chosen as $x_m(t_0) = \bar x$ where 
\be\label{eq:orm1xmo} {\bar x}\triangleq  \frac{-b {\bar r}}{a}>0\ee 

so that $x_m(t)=\xm$ for all time.  
Then the error dynamics are completely described by the second order dynamics
\be\label{eq:orm1e}
\dot z(t; x_0,a, b, \gamma, {\bar x}) =  \bb \dot e(t) \\ \dot \phi(t)  \eb=\bb a e(t) + b \thetae(t) (e(t)+\xm)  \\ - \gamma e(t) (e(t)+\xm) \eb
\ee
with $s(t; t_0, z(t_0))$ the transition function for the dynamics above.

The invariant set is constructed by first defining three 1-dimensional manifolds $ \Su_1, \Su_2, \Su_3$, three preliminary subsets of $\Re^2$ which we will denote $\PP_1,\PP_2,\PP_3$, and  finally three regions $\M_1,\M_2,\M_3$ are defined whose union is our invariant set of interest. Use Figure \ref{fig:e_theta_o} to help visualize these regions. We begin by defining the surface
\be\label{s11}
\Su_1 \triangleq \lb [e,\thetae]^{\mathsf T} \left|\  e = -\xm  \right. \rb. 
\ee
The region $\PP_1 \subset \RR^2$ and the second surface $\Su_2$ are defined as
\begin{align} 
\PP_1 & \triangleq \lb [e,\thetae]^{\mathsf T} \left|\ \thetae < \frac{a}{b} \right. \rb\notag \\ 
\Su_2 & \triangleq \lb [e,\thetae]^{\mathsf T} \left|\ e = \frac{(a - b \thetae)\xm}{a + b \thetae}, [e,\thetae]^{\mathsf T} \in \PP_1 \right. \rb . \label{s12}
\end{align}
Similarly a second subset of the  error-space $\PP_2 \subset \RR^2$ and a third surface $\Su_3$ are defined as
\ben\begin{split} \PP_2 &   \triangleq\lb [e,\thetae]^{\mathsf T} \left|\ \frac{a}{b} \leq \thetae < 0 \right. \rb \\ 
\Su_3 & \triangleq \lb [e,\thetae]^{\mathsf T} \left|\ e = 0, [e,\thetae]^{\mathsf T} \in \PP_2 \right. \rb .\end{split}\een
%
We now define regions $\M_1$ and $\M_2$ as
\begin{align}
\label{m11}\M_1 & \triangleq \lb [e,\thetae]^{\mathsf T} \left|\ -\xm < e < \frac{(a - b \thetae){\bar x}}{a + b \thetae}, [e,\thetae]^{\mathsf T} \in \PP_1 \right. \rb  \\  
\label{m12}\M_2 &  \triangleq \lb [e,\thetae]^{\mathsf T} \left|\ -\xm < e < 0, [e,\thetae]^{\mathsf T} \in \PP_2 \right. \rb. 
\end{align}
%
From these definitions, we note that the surfaces $\Su_1$ and $\Su_2$ form the two sides of the region $\M_1$. Similarly, $\Su_1$ and $\Su_3$ form the sides of the region $\M_2$. In order to complete the invariant set a third region is defined using the Lyapunov function in \eqref{eq:lyaporm} which gives us the convex bounded region
%
%
\be \label{m13}\M_3  \triangleq \lb [e,\thetae]^{\mathsf T} \left|\ e^2 + \frac{1}{\gamma} \thetae^2 <  \xm^2 \right. \rb .\ee
The union of  the three regions is defined as
\be\label{eq:orm1m} \M_0 \triangleq \M_1 \cup \M_2 \cup \M_3. \ee

The following theorem will show three facts. First, the error velocities within $\M_0$ are finite and bounded even though $\M_0$ is unbounded. Second, $\M_0$ is an invariant set. Lastly, a lower limit on the time of convergence is given as a function of the initial condition $z(t_0)$ and the ratio $ \frac{\norm{s(t_1; t_0, z(t_0))}}{\norm{s(t_0; t_0, z(t_0))}}$ for some $t_1\geq t_0$. The conclusion to be arrived at is that the system is UASL and not ESL.


\begin{figure}[t!]
\centering  \label{e_theta}
\psfrag{a}[ccl][ccl][1]{$-\bar x$}
\psfrag{b}[ccl][ccl][1]{$\phi$}
\psfrag{c}[ccl][ccl][1]{$e$}
\psfrag{d}[ccl][ccl][1]{$\M_3$}
\psfrag{e}[ccl][ccl][1]{$\Su_3$}
\psfrag{f}[ccl][ccl][1]{$\M_2$}
\psfrag{g}[ccl][ccl][1]{$\M_1$}
\psfrag{h}[ccl][ccl][1]{$\Su_2$}
\psfrag{i}[ccl][ccl][1]{$\Su_1$}

\includegraphics[width=3in]{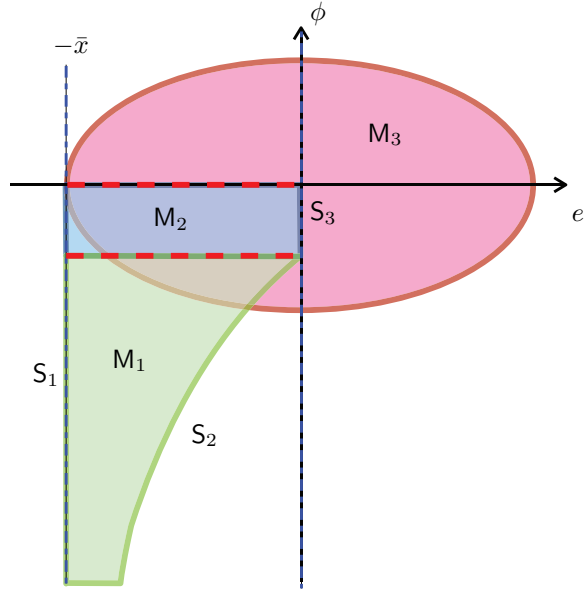}
\caption{The three regions $\M_1$ (green), $\M_2$ (blue) and $\M_3$ (red) whose union results in the invariant set $\M_0$. }\label{fig:e_theta_o}
\end{figure}


\begin{thm}\label{thm:orm} For the error dynamics $z(t)$ with $r(t)={{\bar r}}$ and $x_m(t_0)={\xm}$ with $\M_0$ defined in \eqref{eq:orm1m} and $s(t; t_0, z(t_0))$ the transition function of the differential equation \rep{eq:orm1e}, the following hold
\begin{enumerate}[label={\normalfont (\roman*)}]
\item $\norm{\dot z} \leq d_z$ for all $ z \in \M_0$ where \ben d_z \triangleq \sqrt{(|a {\bar x}| + 2 |b \sqrt{\gamma} {\bar x}^2|)^2 + (2 \gamma {\bar x}^2)^2} \een
\item $\M_0$ is an invariant set.
\item A trajectory beginning at ${z(t_0)\in \M_0}$ will converge to a fraction of its original magnitude at time $t_1$, with
\be T \geq \frac{\norm{z(t_0)}(1-c)}{d_z} \label{eq:Tbound} \ee
where 
$c = \frac{\norm{s(t_1; t_0, z(t_0))}}{\norm{s(t_0; t_0, z(t_0))}}$ and $T = t_1 - t_0$.
\end{enumerate}
\end{thm}

\begin{proof}[Proof of {\normalfont (i)}] From the definition of $\M_1$ in \eqref{m11} and $\M_2$ in \eqref{m12}, and the definition of $\dot \phi$ in \eqref{eq:orm1e} it follows that $\abs{\dot \thetae(z)} \leq \gamma \frac{{\bar x}^2}{4} $ for all ${z\in \M_1 \cup \M_2}$. Similarly, from the definition of $\M_3$  in \eqref{m13} it follows that $\abs{\dot \thetae(z)} \leq 2 \gamma  {\bar x}^2 $ for all ${z\in \M_3}$. Therefore\be \label{phib1}\abs{\dot \thetae(z)} \leq 2 \gamma  {\bar x}^2\ee
for all ${z \in \M_0}$ where $\M_0$ is defined in \eqref{eq:orm1m}.

From the definition of $\dot e$ in \eqref{eq:orm1e} and the definitions of $\M_1$, $\M_2$ and $\M_3$ it follows that $\abs{\dot \thetae(z)} \leq\abs {a {\bar x}}$ for all $z\in \M_1 \cup \M_2$ and $\abs{\dot \thetae(z)} \leq\abs {a {\bar x}} + 2 b\sqrt{\gamma}{\bar x}^2$ for all $z\in \M_3$. Therefore
\be \label{phiz1}\abs{\dot e(z)} \leq \abs {a {\bar x}} + 2 b\sqrt{\gamma} {\bar x}^2\ee
for all $z \in \M_0$. From the bounds in \eqref{phib1} and \eqref{phiz1} for $\dot \phi$ and $\dot e$ respectively Theorem \ref{thm:orm}(i) follows.\let\qed\relax\end{proof}
%
%
%
%
%
\begin{proof}[Proof of {\normalfont (ii)}] In order to evaluate the behavior of the trajectories on the surfaces $\Su_1$, $\Su_2$ and $\Su_3$, normal vectors are defined along the surfaces that point toward $\M_0$. The normal vectors are 
$$ \hat n_1 = [1,0]^{\mathsf T} ,\quad \hat n_2(z) = \left[\frac{-\partial e}{\partial \thetae}, 1 \right]_{z\in {\Su_2}}^{\mathsf T}, 
\quad \text{and}\quad  \hat n_3 = [-1,0]^{\mathsf T} $$
where
$ \frac{\partial e}{\partial \thetae} = \frac{-2 b {\bar x} a}{(a + b\thetae)^2} $. We then find that  $ \hat n_i^{\mathsf T}(z) \dot z (z)  \geq 0$ for ${z \in \Su_i}$ and $i = 1,2,3$.   From the general stability proof of the adaptive system with Lyapunov function $V = e^2 + \frac{1}{\gamma}\thetae^2$ once within $\M_3$ a trajectory cannot leave it. \let\qed\relax\end{proof}

\begin{proof}[Proof of {\normalfont (iii)}] For a trajectory to traverse from $z(t_0)$ to a magnitude less than $c\norm{z(t_0)}$ (such that $\norm{s(t_1)} \leq c\norm{s(t_0)}$) it must travel at least a distance $\norm{z(t_0)}(1-c)$ over which it has a maximum rate of $d_z$ therefore
\ben T \geq \frac{\norm{z(t_0)}(1-c)}{d_z} \label{eq:T1ORM1}. \qedhere\een 
\end{proof}

\begin{proof}[Proof of Theorem \ref{thm:orms}]
The results from Theorem \ref{thm:orm} illustrate that for an input which provides persistent excitation of the reference model, there exists an unbounded region where the adaptive system is UASL and not ESL. For the system to possess ESL, the lower bound in  \eqref{eq:Tbound} needs to be dependent only on $c$ and independent of $z(t_0)$, see Remark \ref{rem:uasl} with $c$ analogous to $\eta/\rho$. The lower bound on $T$ is therefore sufficient to prove that ESL is \benone{ not possible} \bencrossone{ impossible }.\end{proof}

It can also be shown that the learning rate, $\dot \phi$,  of the adaptive parameter  tends to zero as the initial adaptive parameter error $\phi(t_0)$ tends to negative infinity inside $\M_1$. In the previous theorem we only showed that $\dot \phi$ is uniformly bounded for all initial conditions inside the larger set $\M_0$. Thus, not only is ESL impossible, there is an unbounded \bencrossone{ {\em sticking} } region in the base of $\M_1$ where adaptation occurs at a slower and slower rate the deeper the initial condition starts in the trough of $\M_1$. This \bencrossone{ sticking } \bencrossone{ affect } \benone{ effect }is visualized through simulation examples in a later section.

\begin{cor}\label{lem:sticko} For the error dynamics $z(t)$ defined by the differential equation in \eqref{eq:orm1e} with $r(t)={{\bar r}}$ and $x_m(t_0)={\xm}$ it follows that 
$\dot\phi(e,\phi) \to 0$ as $\phi \to -\infty$ with $[e, \phi]^\T \in \M_1$.
\end{cor}
\begin{proof}
For fixed  $\phi$ and an $e$ such that $[e, \phi]^\T\in\M_1$, which we will assume from this point forward in the proof, it follows that $ - {\bar x} \leq e \leq \frac{a-b\phi}{a+b\phi}{\bar x} $ per the definition of $\M_1$ in \eqref{m11}. Written another way, \be\label{eq:pft1} e=-{\bar x}+\Delta \ee where $\Delta \in \left[0,\frac{2a}{a+b\phi}{\bar x} \right]$. Substitution of \eqref{eq:pft1} into the definition of $\dot \phi$ from \eqref{eq:orm1e} it follows that 
$$ \dot \phi = -\gamma (x_{mo} +\Delta)^2 + \gamma {\bar x}({\bar x}-\Delta).$$ After expanding and canceling terms the above equation reduces to \be\label{eq:tt2}\dot \phi = -\gamma \bigl(3 {\bar x} \Delta+\Delta^2\bigr).\ee From the fact that $\Delta \leq \frac{2a}{a+b\phi}{\bar x} $ it follows that $\lim_{\phi\to-\infty}\Delta = 0$ (recall that $a<0$). Using this limiting value of $\Delta$ and \eqref{eq:tt2} it follows that $\lim_{\phi\to-\infty} \dot \phi =0$ when $[e, \phi]^\T \in \M_1$.
\end{proof}

This corollary helps connect the results from this section back to our definitions of PE and PE$^*$, and Remark \ref{rem:1}. While it is possible for $x_m\in\mathrm{PE}$ our analysis technique only allowed us to conclude that $x\in\mathrm{PE}^*$. This was characterized by the fact that in order for $x$ to maintain the same level of excitation, which we referred to as $\alpha^\prime$ in \eqref{eq:pxm}, the time window over which the excitation was measured, $ pT$ in \eqref{eq:pxm}, would have to increase as the norm of the initial conditions of the system increased. This is precisely what is occurring in the bottom of $\M_1$. In the bottom of this region it follows by definition that $\abs{x} \leq \Delta$ which tends to zero as $\phi(t_0)$ decreases to negative infinity, all the while the speed at which the state can leave this region is decreasing as well. 


\subsection{Scalar CRM Adaptive System}

We now consider a modified adaptive system in which the reference model contains a feedback loop with the state error. The plant is the same as that in \eqref{eq:orm1} with and identical control law and the same update law as that in \eqref{eq:orm3}. The reference model however is now defined as 
\be \dot x_m(t) = a x_m(t) + b r(t) - \ell e(t) \label{eq:ref_crm}\ee
where $\ell<0$ and ${\bar x}$ no longer. 
%
%
Throughout this section it is assumed that ${\bar r} > 0$ is a constant, however, no longer does  $x_m(t) = {\bar x}$ for all time. Unlike in the ORM cases, the reference model dynamics cannot be ignored. The resulting system can be represented as
\begin{align}
\dot z(t; x_0,a, b, \gamma, {\bar r},\ell) =  \left[ \begin{array}{l}
\dot x_m(t) \\
\dot e(t)\\
\dot \phi(t)
\end{array}\right] = \left[ \begin{array}{l}
a x_m(t) + b {\bar r} -\ell e(t)\\
(a+\ell) e(t) + b \phi(t) x(t) \\
-\gamma e(t) x(t)
\end{array}\right]. \label{eq:CRM1}
\end{align}
We will show that this modified adaptive system cannot be ESL and for the specific $r(t)$ chosen is UASL.

\begin{thm}\label{thm:stab_crm}
For the system defined in equation \rep{eq:CRM1} with  $r(t) \equiv \overline r$, $\overline r\neq 0$ the equilibrium of $z$ is UASL.\end{thm}
\begin{proof}
Consider the Lyapunov candidate in \eqref{eq:lyaporm} and differentiating along the dynamics in \eqref{eq:CRM1} it follows that $\dot V(e,\thetae) = 2(a+\ell) e^2$.
Since $V > 0$ for all $z\neq 0$, $\dot V \leq 0$ for all $[e \; \thetae]^{\mathsf T} \in \RR^2$ and $V \to \infty \; as \;z \to \infty$, it follows that $z=[{\bar x}, 0, 0]^\T$ is uniformly stable in the large. Since the system is autonomous it follows from the invariance principle that $z=[{\bar x}, 0, 0]^\T$ is UASL as well.
\end{proof}

Now a number of  regions in the state-space ($\RR^3$) are defined which allow the construction and proof of this subsection's main result which mirrors the results of Theorem \ref{thm:orm}. In particular, three regions will be defined. It will then be shown that a specific region $\M_0$, the union of these three regions, will remain invariant. As this region $\M_0$ is infinite and the vector field defined by \rep{eq:CRM1} has a finite maximum velocity, we can conclude that CRM adaptive systems do not posses exponential stability in the large but are at best UASL.

Define a subset of the state-space, $\PP_1 \subset \RR^3$
$$\PP_1 \triangleq \left\{ [x_m,e,\thetae]^{\mathsf T} \left| \thetae < \frac{a + \ell}{b}, \frac{b{\bar r}}{a + \ell} \leq x_m \leq {\bar x}  ,[x_m,e,\thetae]^{\mathsf T} \in \RR^3 \right.\right\}$$
and within the subset $\PP_1$ a region
$$\M_1 \triangleq \left\{ [x_m,e,\thetae]^{\mathsf T} \left| -x_m \leq e \leq \frac{x_m (a + \ell + b\thetae)}{a+  \ell -b\thetae}, [x_m,e,\thetae]^{\mathsf T} \in \PP_1 \right. \right\}.$$
Define a second subset of the state-space, $\PP_2 \subset \RR^3$
$$\PP_2 \triangleq \left\{ [x_m,e,\thetae]^{\mathsf T} \left| \frac{a + \ell}{b} \leq \thetae < 0, \frac{b{\bar r}}{a+ \ell} \leq x_m \leq {\bar x}  , [x_m,e,\thetae]^{\mathsf T} \in \RR^3 \right. \right\}$$
and within this subset a region
$$\M_2 \triangleq \left\{ [x_m,e,\thetae]^{\mathsf T} \left| -x_m \leq e \leq 0, [x_m,e,\thetae]^{\mathsf T} \in \PP_2 \right. \right\}.$$
A third region is defined as
$$\M_3 \triangleq \left\{ [x_m,e,\thetae]^{\mathsf T} \left|  e^2 + \frac{1}{\gamma}\thetae^2 \leq {\bar x}^2, 0\leq x_m \leq 2 {\bar x}, [x_m,e,\thetae]^{\mathsf T} \in \RR_3 \right. \right\}$$
The union of these three $\M$ regions is then the invariant set $\M_0$, defined as
\be\label{eq:crm1m} \M_0 \triangleq \M_1 \cup \M_2 \cup \M_3 \ee
The three regions are shown in Figure \ref{fig:e_theta}. Four surfaces of this region will be used in the proof of the following theorem
\begin{align}\Su_1 &\triangleq \left\{ [x_m,e,\thetae]^{\mathsf T} \left| e = \frac{x_m (a+ \ell + b\thetae)}{a+ \ell -b\thetae}, [x_m,e,\thetae]^{\mathsf T} \in \PP_1 \right. \right\}\label{s21}\\
\Su_2 &\triangleq \left\{ [x_m,e,\thetae]^{\mathsf T} \left| e = -x_m, [x_m,e,\thetae]^{\mathsf T} \in \PP_1 \cup \PP_2 \right. \right\}\label{s22}\\
\Su_3 &\triangleq \left\{ [x_m,e,\thetae]^{\mathsf T} \left| e = 0, [x_m,e,\thetae]^{\mathsf T} \in \PP_2 \right. \right\}\notag\\
\Su_4 &\triangleq \left\{ [x_m,e,\thetae]^{\mathsf T} \left|  e^2 + \frac{1}{\gamma}\thetae^2 = {\bar x}^2, [x_m,e,\thetae]^{\mathsf T} \in \M_0 \right. \notag\right\}\end{align}
\begin{figure}[t!]
\centering  \label{e_theta}
\begin{minipage}{.5\textwidth}
  \centering
  \includegraphics[width=2.75in]{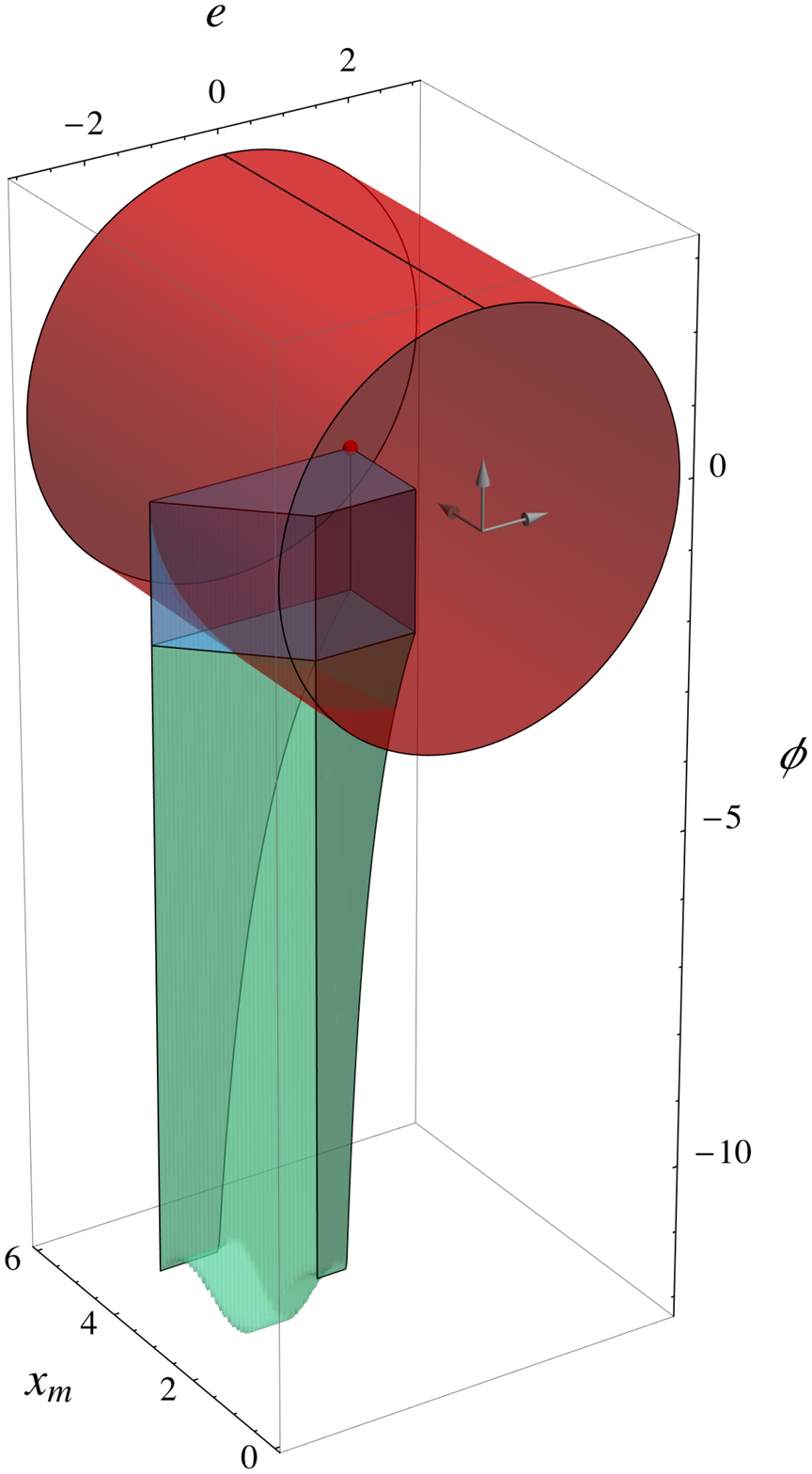}
  \label{fig:test1}
\end{minipage}%
\begin{minipage}{.5\textwidth}
  \centering
  \includegraphics[width=2.75in]{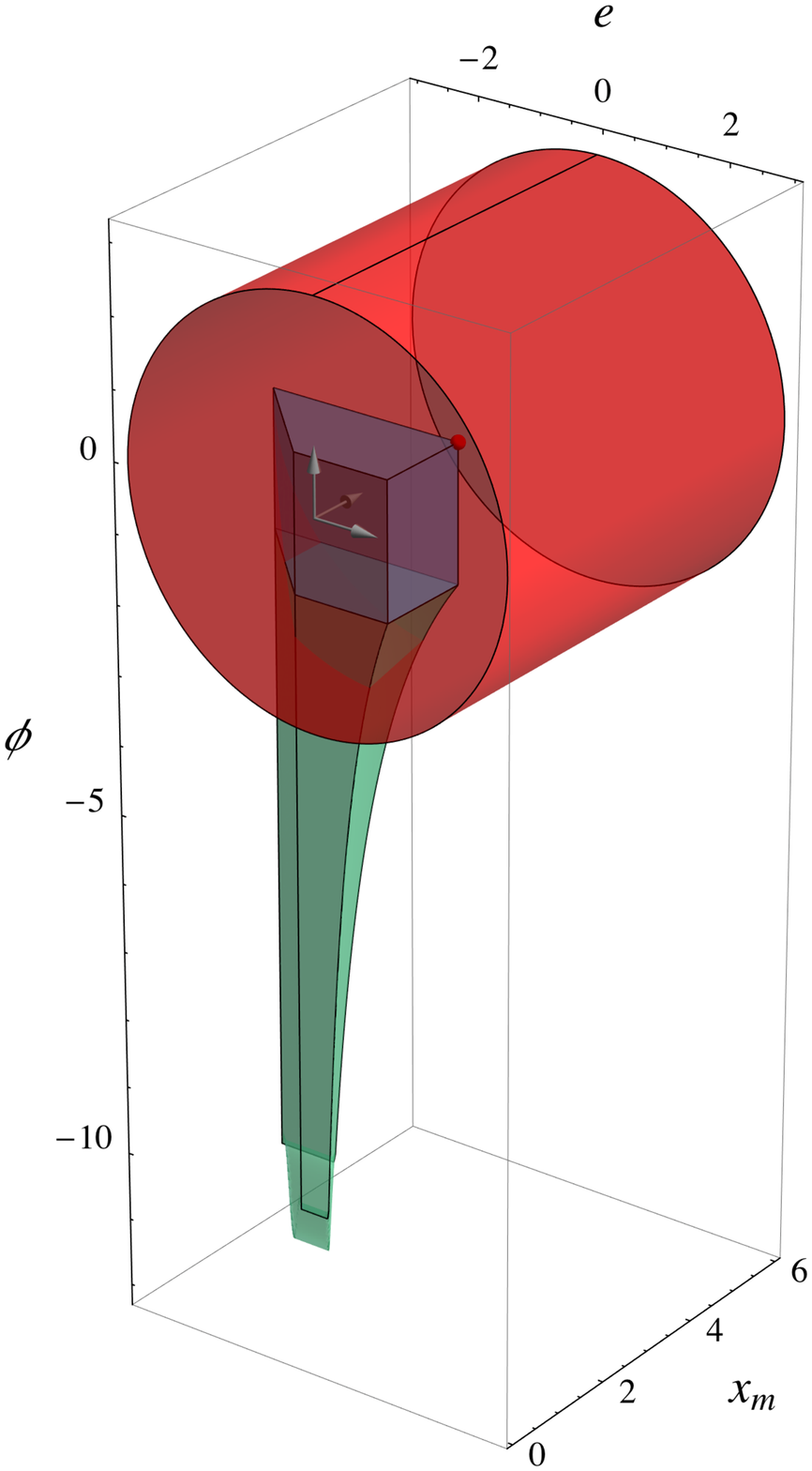}
  \label{fig:test2}
\end{minipage}
\caption{The three regions $\M_1$ (green), $\M_2$ (blue) and $\M_3$ (red) whose union results in the invariant set $\M_0$. }\label{fig:e_theta}
\end{figure}

%



\begin{thm} For the error dynamics $z(t)$ with $r(t)={{\bar r}}$, $\M_0$ as defined in \eqref{eq:crm1m}, and $s(t; t_0, z(t_0))$ the transition function of the differential equation \rep{eq:CRM1}, the following hold
\begin{enumerate}[label={\normalfont (\roman*)}]
\item $\norm{\dot z} \leq d_z$ for all $ z \in \M_0$ where \
\ben d_z \triangleq \sqrt{(\abs{(a+ \ell) {\bar x}} + 2 b \sqrt{\gamma}{\bar x}^2)^2 + (2\gamma {\bar x}^2)^2 + (\abs{(a+\ell) {\bar x}} +\bar r)^2}\een
\item $\M_0$ is an invariant set.
\item A trajectory beginning at ${z(t_0)\in \M_0}$ will converge to a fraction of its original magnitude at time $t_1$, with
\be T \geq \frac{\norm{z(t_0)}(1-c)}{d_z} \label{eq:Tbound} \ee
where 
$c = \frac{\norm{s(t_1; t_0, z(t_0))}}{\norm{s(t_0; t_0, z(t_0))}}$ and $T = t_1 - t_0$.
\end{enumerate}
\end{thm}

\begin{proof}[Proof of {\normalfont (i)}]Each component of the vector field is bounded:
\begin{align*} 
\abs{\dot \thetae(z)} &\leq 2\gamma x_0^2 \\
\abs{\dot e(z)} &\leq \abs{(a+\ell){\bar x}} + 2 b \sqrt{\gamma} \bar x^2  \\
\abs{\dot x_m(z)} &\leq \abs{(a+\ell){\bar x}}+ \bar r 
\end{align*}
when $z \in \M_0$, and thus $\norm{\dot z} \leq d_z$.
\let\qed\relax\end{proof}

\begin{proof}[Proof of {\normalfont (ii)}]In order to evaluate the behavior of the trajectories on the surfaces of $\M_1$ and $\M_2$, normal vectors are defined along the surfaces. The normal vectors $\hat n_2$ and $\hat n_3$ have trivial definitions easily determined by inspection. The normal vector $\hat n_1$ is constructed using the cross product of two tangential vectors $\hat n_1 = \hat t_1 \otimes \hat t_2$
where
$$ \hat t_1 = \left[1 \;\; 0 \;\; \frac{\partial e}{\partial x_m} \right]^{\mathsf T}_{z\in \Su_1} \quad \text{and} \quad \hat t_2 = \left[1 \;\; \frac{\partial e}{\partial x_m} \;\; 0\right]^{\mathsf T}_{z\in \Su_1}.$$
It follows directly that ${ \hat n_i^{\mathsf T}(z) \dot z (z)  \geq 0}$ for ${z \in \Su_i}$ and $i = 1,2,3$. From the stability analysis in the proof of Theorem \ref{thm:stab_crm} we know that $\Su_4$ is simply a level set of the Lyapunov function and thus $\M_4$ is invariant. Therefore no trajectory can exit $\M_0$ making it an invariant set.
\let\qed\relax\end{proof}

\begin{proof}[Proof of {\normalfont (iii)}] Identical to the proof of item (iii) in Theorem \ref{thm:orm}. \end{proof}

Just as with an ORM, with a CRM the dynamics are at best UASL. The \benone{ region of slow convergence} \bencrossone{ sticking regime } is present in  CRM adaptive control as well and a similar corollary holds.
\begin{cor}\label{cor2} For the error dynamics $z(t)$ defined by the differential equation in \eqref{eq:CRM1} with $r(t)={{\bar r}}$ it follows that 
$\dot\phi(x_m,e,\phi) \to 0$ as $\phi \to -\infty$ with $[x_m,e, \phi]^\T \in \M_1$.
\end{cor}

\section{Simulation examples}
Simulations are now presented for the ORM adaptive system and the  CRM adaptive system. The main purpose of these simulations is to illustrate the invariance of their respective $\M_0$, and the slow convergence, especially the \benone{ sluggish} \bencrossone{ sticking } phenomenon that is treated in Corollaries \ref{lem:sticko} and \ref{cor2}. Before continuing to the results we need to distinguish between the surfaces in the ORM and CRM cases and define two new surfaces. First, let the following two surfaces in the ORM case be redefined as $\Su_{O1}= \Su_1$ and $\Su_{O2} = \Su_2$ where $\Su_1$ and $\Su_2$ are defined in \eqref{s11} and \eqref{s12}. Similarly for the CRM, $\Su_{C1}= \Su_1$ and $\Su_{C2} = \Su_2$ where $\Su_1$ and $\Su_2$ are defined in \eqref{s21} and \eqref{s22}. The two new surfaces to be defined pertain to the condition $\dot e=0$. For ORMs this surface is defined as 
\ben
\Su_{O5}  \triangleq \lb [e,\thetae]^{\mathsf T} \left|\ e = \frac{-\xm b \phi}{a + b \thetae} \right. \rb 
\een
and for the CRMs a similar curve is defined as
\ben
\Su_{C5}  \triangleq \lb [e,\thetae]^{\mathsf T} \left|\ e = \frac{-x_m b \phi}{a + b \thetae}, x_m=\frac{\ell e - b {\bar r}}{a} \right. \rb 
\een
where the second equation in the definition of $\Su_{C5}$ is derived from \eqref{eq:ref_crm} by setting $\dot x_m =0$.
Nine initial states are chosen specifically for each system, defined in Tables \ref{ICsORM} and \ref{ICsCRM}. Rather than defining numerical values for each initial condition, we choose them as points of intersection between two unique surfaces. The values of the parameters for the simulations are as follows
\be \label{eq:examplesys}
  \begin{array}{ccccccccc}
    a = -1, & \; & \ell = -1, & \; & \gamma = 1, & \; &
    b = 1,  &\; & r = 3.   
  \end{array}
\ee


\begin{table}
\begin{center}
\begin{tabular}{|c|c|c|c|}
  \hline
   & $\Su_{O1}$ & $\Su_{O5}$ & $\Su_{O2}$ \\ \hline
  $\thetae=-2$ & $z_1$ & $z_4$ & $z_7$ \\ \hline
  $\thetae=-4$ & $z_2$ & $z_5$ & $z_8$ \\ \hline
  $\thetae=-8$ & $z_3$ & $z_6$ & $z_9$ \\
  \hline
\end{tabular}
\end{center}
\caption{Initial conditions $z_i, i=1,2, \ldots, 9$ for the ORM example system. Each initial conditions, $z_i$, is the point of intersection of the two indicated surfaces in the corresponding row and column.} \label{ICsORM}
\end{table}
\begin{table}
\begin{center}
\begin{tabular}{|c|c|c|c|}
  \hline
    & $\Su_{C1}$ & $\Su_{C5}$ & $\Su_{C2}$ \\ \hline
  $\thetae=-2$ & $z_1$ & $z_4$ & $z_7$ \\ \hline
  $\thetae=-4$ & $z_2$ & $z_5$ & $z_8$ \\ \hline
  $\thetae=-8$ & $z_3$ & $z_6$ & $z_9$ \\
  \hline
\end{tabular}
\end{center}
\caption{Initial conditions $z_i, i=1,2, \ldots, 9$ for the CRM example system. Each initial conditions, $z_i$, is the point of intersection of the two indicated surfaces in the corresponding row and column.} \label{ICsCRM}
\end{table}

\begin{figure}[t!]
\begin{subfigmatrix}{2}
\subfigure[2D phase portrait for the ORM adaptive system with  initial conditions defined in table \ref{ICsORM}]{\psfrag{s}[cc][cc][.7]{$e$}
\psfrag{a}[Bl][Bl][.7]{$z_1$}
\psfrag{b}[Bl][Bl][.7]{$z_2$}
\psfrag{c}[Bl][Bl][.7]{$z_3$}
\psfrag{d}[tl][tl][.7]{$z_4$}
\psfrag{e}[tl][tl][.7]{$z_5$}
\psfrag{f}[tl][tl][.7]{$z_6$}
\psfrag{g}[Bl][Bl][.7]{$z_7\;$}
\psfrag{h}[Bl][Bl][.7]{$z_8\;$}
\psfrag{i}[cl][cl][.7]{$z_9\;$}
\psfrag{j}[Bl][Bl][.7]{$\Su_{O1}\;$}
\psfrag{k}[tl][bl][.7]{$\;\Su_{O2}$}
\psfrag{l}[bc][bc][.7]{$\Su_{O5}\;\;$}
\psfrag{t}[cc][cc][.8]{$\phi$}
\includegraphics[width=2.2in]{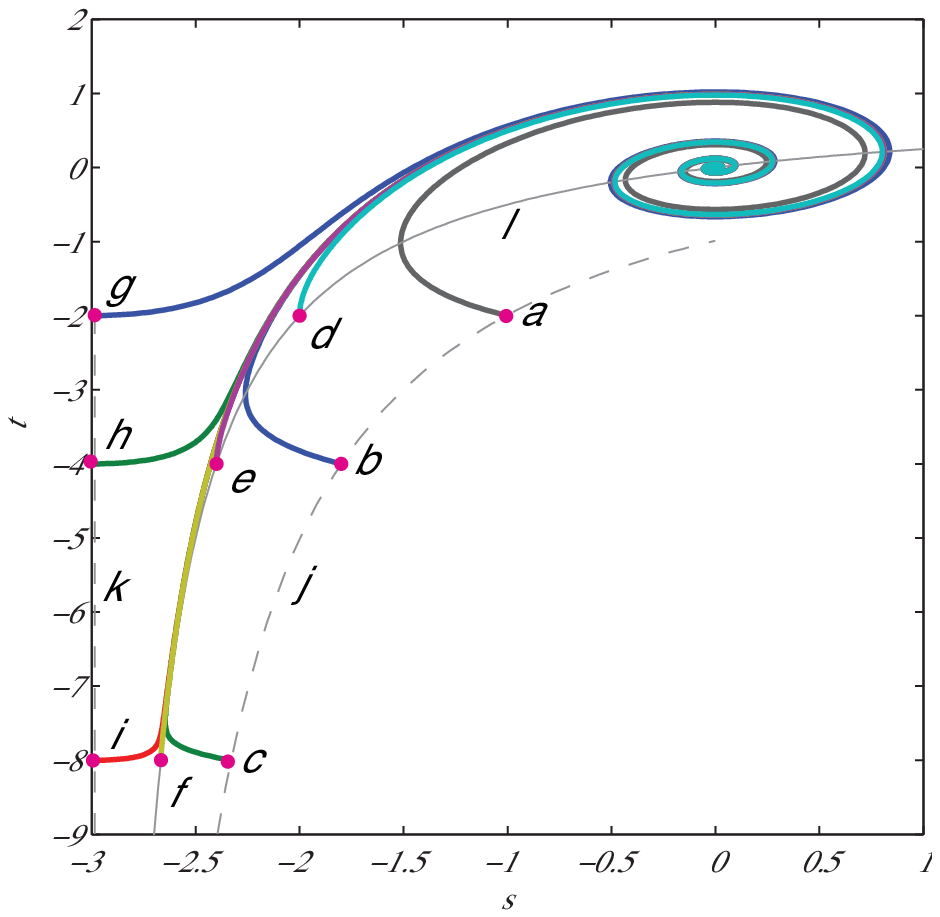}
\label{fig:ppORM}
}
\subfigure[2D projection of the 3D phase portrait for the CRM adaptive system with initial conditions defined in table \ref{ICsCRM}]{\psfrag{s}[cc][cc][.7]{$e$}
\psfrag{a}[Bl][Bl][.7]{$z_1$}
\psfrag{b}[Bl][Bl][.7]{$z_2$}
\psfrag{c}[Bl][Bl][.7]{$z_3$}
\psfrag{d}[tl][tl][.7]{$z_4$}
\psfrag{e}[tl][tl][.7]{$z_5$}
\psfrag{f}[tl][tl][.7]{$z_6$}
\psfrag{g}[cr][cr][.7]{$z_7\;$}
\psfrag{h}[cr][cr][.7]{$z_8\;$}
\psfrag{i}[cr][cr][.7]{$z_9\;$}
\psfrag{j}[Bl][Bl][.7]{$\Su_{C1}\;$}
\psfrag{k}[Br][Br][.7]{$\Su_{C2}\;$}
\psfrag{l}[Bc][Bc][.7]{$\Su_{C5}\;$}
\psfrag{t}[cc][cc][.8]{$\phi$}
\includegraphics[width=2.2in]{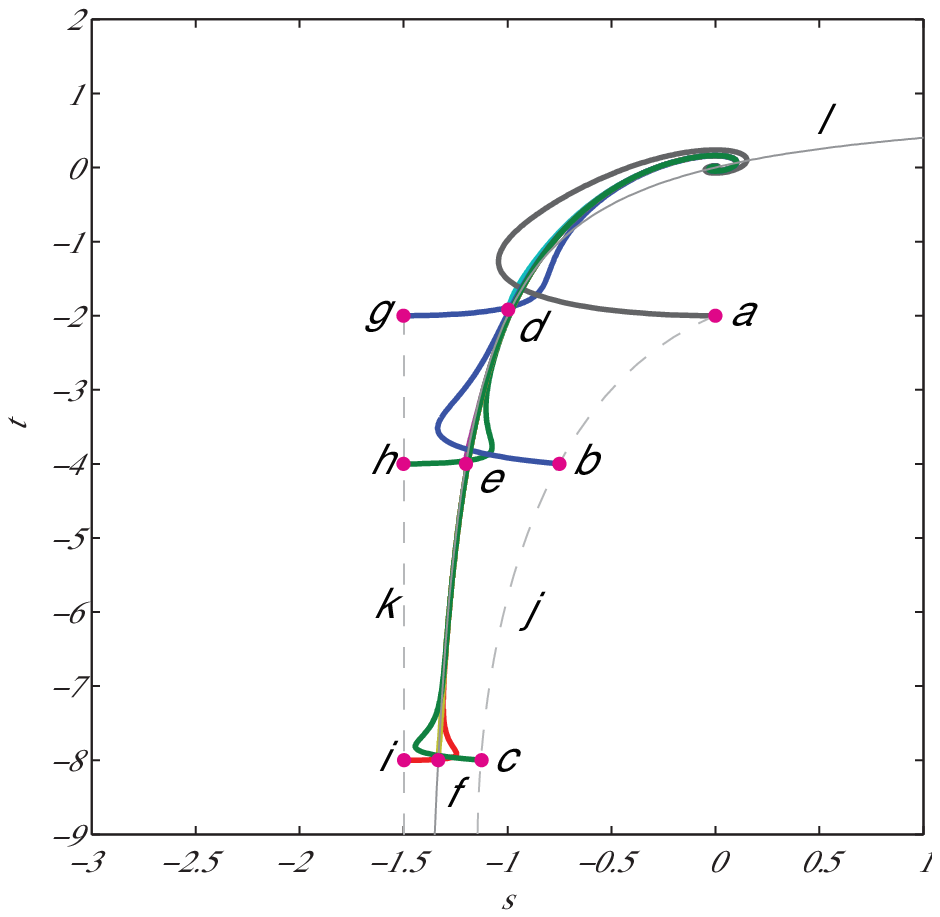}
\label{fig:ppCRM}
}
\end{subfigmatrix}
\caption{Phase portraits of the ORM and CRM adaptive systems}
\end{figure}

\begin{figure}[t!]
\begin{subfigmatrix}{2}
\subfigure[Simulation of the the ORM adaptive system.]{\psfrag{l100000000}[Bl][Bl][.7]{$z(0) = z_4$}
\psfrag{l200000000}[Bl][Bl][.7]{$z(0) = z_5$}
\psfrag{l300000000}[Bl][Bl][.7]{$z(0) = z_6$}
\psfrag{e}[Bl][Bl][.7]{$e$}
\psfrag{xm}[Bl][Bl][.7]{$x_m$}
\psfrag{xp}[Bl][Bl][.7]{$x$}
\psfrag{t}[Bl][Bl][.7]{$\thetae$}
\psfrag{time}[Bc][Bc][.7]{$t$}
\includegraphics[width=2.3in]{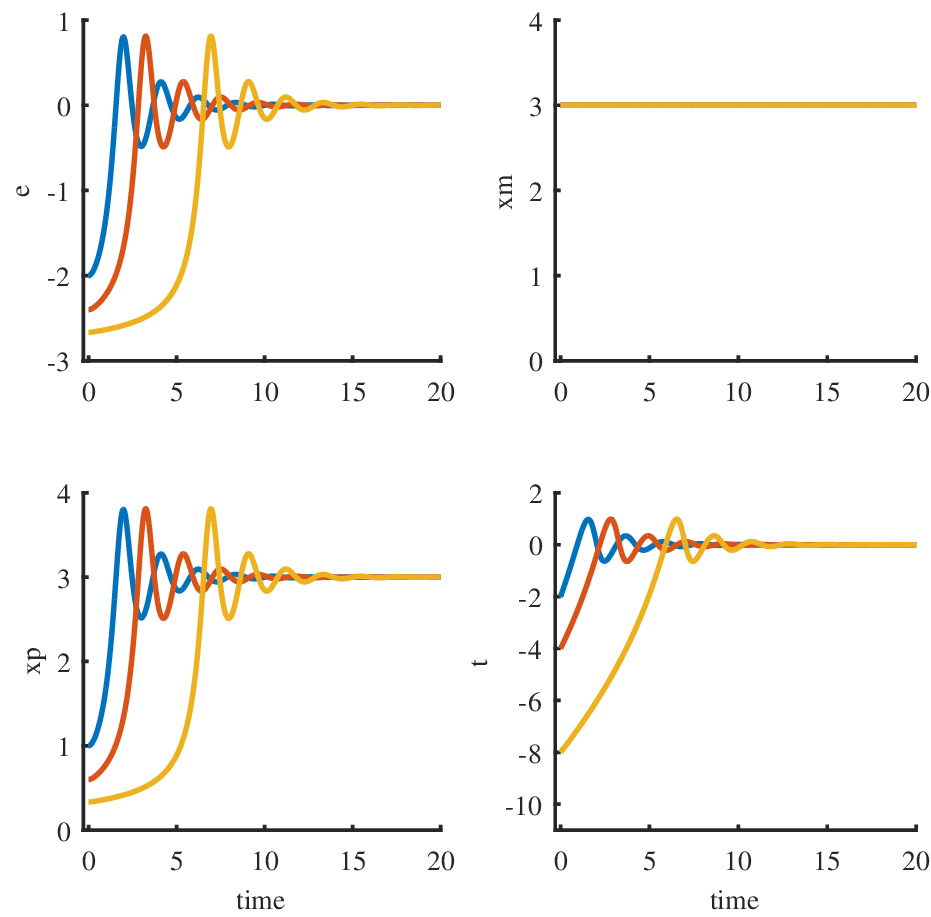}
\label{fig:timeORM}
}
\subfigure[Simulation of the the CRM adaptive system.]{\psfrag{l100000000}[Bl][Bl][.7]{$z(0) = z_4$}
\psfrag{l200000000}[Bl][Bl][.7]{$z(0) = z_5$}
\psfrag{l300000000}[Bl][Bl][.7]{$z(0) = z_6$}
\psfrag{e}[Bl][Bl][.7]{$e$}
\psfrag{xm}[Bl][Bl][.7]{$x_m$}
\psfrag{xp}[Bl][Bl][.7]{$x$}
\psfrag{t}[Bl][Bl][.7]{$\thetae$}
\psfrag{time}[Bc][Bc][.7]{$t$}
\includegraphics[width=2.3in]{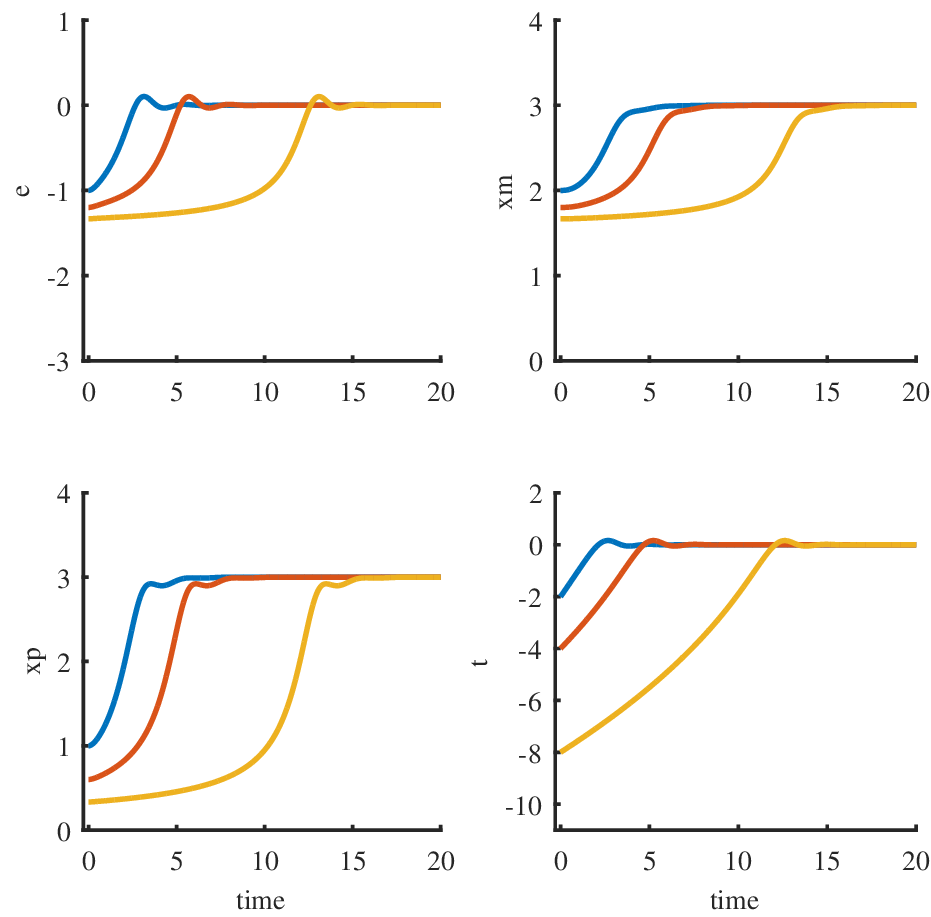}
\label{fig:timeCRM}
}
\end{subfigmatrix}
\caption{Time series trajectories of the ORM and CRM adaptive systems for intial conditions $z_4$ ({\color{ml1}blue}), $z_5$ ({\color{ml2}orange}), and $z_6$ ({\color{ml3}yellow}) as defined in Table \ref{ICsORM} and \ref{ICsCRM}.} \label{timeORM1}
\end{figure}
Figure \ref{fig:ppORM} contains the 2-dimensional phase portrait for trajectories of the ORM adaptive system resulting from each of the initial conditions of Table \ref{ICsORM}. Figure \ref{fig:ppCRM} contains the 2-dimensional projection of the 3-dimensional phase space trajectories of the CRM adaptive system resulting from each of the initial conditions of Table \ref{ICsCRM}. Before we proceed, we observe that in both Figures \ref{fig:ppORM} and \ref{fig:ppCRM}, there is an attractor that all initial conditions converge to. This attractor partially coincides with $\Su_{O5}$ and $\Su_{C5}$. We focus on those initial conditions that are closest to these attractors that are common to both ORM  and CRM adaptive systems, which are given by initial conditions $z_4$, $z_5$, and $z_6$. With these initial conditions we next discuss the \benone{ region of slow convergence } \bencrossone{ sticking regime  }in both adaptive systems.

We present time responses of $e$, $\thetae$, $x$, and $x_m$ for the ORM system in Figure \ref{fig:timeORM} and  the CRM adaptive system in Figure \ref{fig:timeCRM}, respectively, for the initial conditions $z_4$, $z_5$, and $z_6$. Defining $T_s$ as the settling time beyond which $\norm{z(t_0)-z(\infty)}$ reduces to 5\% of its initial value, we have that $T_s\in\{5.37,\;5.62,\;8.19\}$ for these three initial conditions for the ORM system and $T_s\in\{3.69,\;5.85,\;12.74\}$ for the CRM system. Notice that although $\frac{z(t_0)}{z(t_0+T_s)}$ is identical for all three trajectories, $T_s$ increases as $\norm{z(t_0)}$ increases, implying that the system is not exponentially stable in the large. 

Trajectories initialized at both $z_5$ and $z_6$ demonstrate the \benone{ slow convergence} \bencrossone{ {\em sticking} property } described in this paper, which is characterized by the nearly flat portion of the response of $e$ and $x$ prior to convergence.
From the third initial condition, $z_6$, the exacerbated \bencrossone{ sticking } \benone{ sluggish }effect in the CRM adaptive system can clearly be seen. The error convergence for large initial conditions is even slower compared to that of the ORM system. It was observed that this convergence became slower as $\abs{\ell}$ was increased further. It should be noted that these convergence properties co-exist with the absence of the oscillatory behavior in the CRM in comparison to the ORM. That is, the introduction of the feedback gain $\ell$ helps in producing a {\em smooth} adaptation, but not a {\em fast} adaptation. Increasing $\gamma$ along with $\ell$ can keep convergence times similar to those of the ORM while maintaining reduced oscillations.

\section{Conclusions}

\bencrossone{ Precise definitions of exponential stability along with initial condition dependent definitions of persistence of excitation have been given. With these definitions it has been shown that when the reference model is persistently exciting the direct adaptive control problem can at best be uniformly asymptotically stable in the large. This is illustrated for both ORM and CRM adaptive systems. } 

\benone{ In this paper, precise definitions of asymptotic and exponential stability are reviewed and a definition of weak persistent excitation is introduced, which is initial-condition dependent. With these definitions it has been shown that when persistent excitation conditions are imposed on the reference model, it results in weak persistent excitation of the adaptive system. The implication of this weak PE is that the speed of convergence is initial condition dependent, resulting in UASL of the origin in the underlying error system. Exponential stability \textit{in the large }can not be proven and claims of robustness should be based on the UASL property.}

\bibliography{master_2}

\appendix
\section{Definitions}

\begin{defn}[Piecewise smooth function \cite{yua77}] \label{def:P}Let $\mathcal C_\delta$ be a set of points in $[t_0,\infty)$ for which there exists a $\delta>0$ such that for all $t_1,t_2\in\mathcal C_\delta$, $t_1\neq t_2$ implies $\abs{t_1 - t_2}\geq \delta$. Then $\mathcal P_{[t_0,\infty)}$ is defined as the class of real valued functions on $[t_0,\infty)$ such that for every $u\in \mathcal P_{[t_0,\infty)}$, there corresponds some $\delta$ and $\mathcal C_\delta$ such that
\begin{enumerate}[label={\normalfont (\roman*)}]
\item $u(t)$ and $\dot u(t)$ are continuous and bounded on $[t_0,\infty) \setminus \mathcal C_\delta$ and
\item for all $t_1\in\mathcal C_\delta$, $u(t)$ and $\dot u(t)$ have finite limits as $t \to t_1^+$ and $t\to t_1^-$
\end{enumerate}
\end{defn}

\end{document}